\documentclass[a4paper,12pt,reqno]{amsart}

\usepackage[T1]{fontenc}
\usepackage[utf8]{inputenc}
\usepackage{lmodern}
\usepackage[english]{babel}

\usepackage[%
	left=2.5cm,       
	right=2.5cm,      
	top=3.5cm,        
	bottom=3.5cm,     
	heightrounded,    
	bindingoffset=0mm 
]{geometry}

\usepackage{amsmath}
\usepackage{amssymb}
\usepackage{mathtools}
\usepackage{mathrsfs}

\numberwithin{equation}{section}

\usepackage{hyperref} 
\usepackage{bookmark}
\usepackage[noabbrev,capitalize]{cleveref}

\usepackage[initials,
]{amsrefs}
\renewcommand{\MR}[1]{} 

\usepackage{enumitem}

\theoremstyle{plain}
	\newtheorem{theorem}{Theorem}[section]
	\newtheorem{lemma}[theorem]{Lemma}
	\newtheorem{proposition}[theorem]{Proposition}
	\newtheorem{corollary}[theorem]{Corollary}

\theoremstyle{definition}
	\newtheorem{definition}[theorem]{Definition}
	\newtheorem{example}[theorem]{Example}
	\newtheorem{remark}[theorem]{Remark}

\newcommand{\N}{\mathbb{N}}
\newcommand{\R}{\mathbb{R}}
\newcommand{\de}{\partial}
\newcommand{\eps}{\varepsilon}
\newcommand{\closure}[2][3]{%
  {}\mkern#1mu\overline{\mkern-#1mu#2}}
\newcommand{\weakto}{\rightharpoonup}

\newcommand{\res}{\mathop{\hbox{\vrule height 7pt width .5pt depth 0pt
\vrule height .5pt width 6pt depth 0pt}}\nolimits}
\newcommand{\Haus}[1]{\mathscr{H}^{#1}}
\newcommand{\Leb}[1]{\mathscr{L}^{#1}} 
\newcommand{\redb}{\mathscr{F}}
\newcommand{\M}{\mathcal{M}}
\newcommand{\DM}{\mathcal{DM}}
\newcommand{\di}{\,\mathrm{d}}
\newcommand{\divm}{{div}}

\makeatletter
\newcommand*{\mint}[1]{%
  \mint@l{#1}{}%
}
\newcommand*{\mint@l}[2]{%
  \@ifnextchar\limits{%
    \mint@l{#1}%
  }{%
    \@ifnextchar\nolimits{%
      \mint@l{#1}%
    }{%
      \@ifnextchar\displaylimits{%
        \mint@l{#1}%
      }{%
        \mint@s{#2}{#1}%
      }%
    }%
  }%
}
\newcommand*{\mint@s}[2]{%
  \@ifnextchar_{%
    \mint@sub{#1}{#2}%
  }{%
    \@ifnextchar^{%
      \mint@sup{#1}{#2}%
    }{%
      \mint@{#1}{#2}{}{}%
    }%
  }%
}
\def\mint@sub#1#2_#3{%
  \@ifnextchar^{%
    \mint@sub@sup{#1}{#2}{#3}%
  }{%
    \mint@{#1}{#2}{#3}{}%
  }%
}
\def\mint@sup#1#2^#3{%
  \@ifnextchar_{%
    \mint@sup@sub{#1}{#2}{#3}%
  }{%
    \mint@{#1}{#2}{}{#3}%
  }%
}
\def\mint@sub@sup#1#2#3^#4{%
  \mint@{#1}{#2}{#3}{#4}%
}
\def\mint@sup@sub#1#2#3_#4{%
  \mint@{#1}{#2}{#4}{#3}%
}
\newcommand*{\mint@}[4]{%
  \mathop{}%
  \mkern-\thinmuskip
  \mathchoice{%
    \mint@@{#1}{#2}{#3}{#4}%
        \displaystyle\textstyle\scriptstyle
  }{%
    \mint@@{#1}{#2}{#3}{#4}%
        \textstyle\scriptstyle\scriptstyle
  }{%
    \mint@@{#1}{#2}{#3}{#4}%
        \scriptstyle\scriptscriptstyle\scriptscriptstyle
  }{%
    \mint@@{#1}{#2}{#3}{#4}%
        \scriptscriptstyle\scriptscriptstyle\scriptscriptstyle
  }%
  \mkern-\thinmuskip
  \int#1%
  \ifx\\#3\\\else_{#3}\fi
  \ifx\\#4\\\else^{#4}\fi  
}
\newcommand*{\mint@@}[7]{%
  \begingroup
    \sbox0{$#5\int\m@th$}%
    \sbox2{$#5\int_{}\m@th$}%
    \dimen2=\wd0 %
    \let\mint@limits=#1\relax
    \ifx\mint@limits\relax
      \sbox4{$#5\int_{\kern1sp}^{\kern1sp}\m@th$}%
      \ifdim\wd4>\wd2 %
        \let\mint@limits=\nolimits
      \else
        \let\mint@limits=\limits
      \fi
    \fi
    \ifx\mint@limits\displaylimits
      \ifx#5\displaystyle
        \let\mint@limits=\limits
      \fi
    \fi
    \ifx\mint@limits\limits
      \sbox0{$#7#3\m@th$}%
      \sbox2{$#7#4\m@th$}%
      \ifdim\wd0>\dimen2 %
        \dimen2=\wd0 %
      \fi
      \ifdim\wd2>\dimen2 %
        \dimen2=\wd2 %
      \fi
    \fi
    \rlap{%
      $#5%
        \vcenter{%
          \hbox to\dimen2{%
            \hss
            $#6{#2}\m@th$%
            \hss
          }%
        }%
      $%
    }%
  \endgroup
}

\renewcommand{\div}{\mathrm{div}}
\renewcommand{\phi}{\varphi}
\renewcommand{\rho}{\varrho}
\renewcommand{\theta}{\vartheta}

\DeclareMathOperator{\supp}{supp}

\DeclareMathOperator{\Lip}{Lip}
\DeclareMathOperator{\loc}{loc}

\DeclarePairedDelimiter{\set}{\{}{\}}
\DeclarePairedDelimiter{\abs}{|}{|}

\mathchardef\ordinarycolon\mathcode`\:
\mathcode`\:=\string"8000
\begingroup \catcode`\:=\active
  \gdef:{\mathrel{\mathop\ordinarycolon}}
\endgroup

\allowdisplaybreaks

\usepackage{color}

\begin{document}

\title[Fractional divergence-measure fields]{Fractional divergence-measure fields, Leibniz rule and Gauss--Green formula}

\author[G.~E.~Comi]{Giovanni E. Comi}
\address[G.~E.~Comi]{Dipartimento di Matematica, Università di Bologna, Piazza di Porta San Donato 5, 40126 Bologna (BO), Italy}
\email{giovannieugenio.comi@unibo.it}

\author[G.~Stefani]{Giorgio Stefani}
\address[G.~Stefani]{Scuola Internazionale Superiore di Studi Avanzati (SISSA), via Bonomea~265, 34136 Trieste (TS), Italy}
\email{giorgio.stefani.math@gmail.com or gstefani@sissa.it}

\thanks{\textit{Acknowledgments}.
The authors are members of  the Istituto Nazionale di Alta Matematica (INdAM), Gruppo Nazionale per l'Analisi Matematica, la Probabilità e le loro Applicazioni (GNAMPA), and are partially supported by the INdAM--GNAMPA 2023 Project \textit{Problemi variazionali per funzionali e operatori non-locali}, codice CUP\_E53\-C22\-001\-930\-001.
The first-named author is partially supported by the INdAM--GNAMPA 2022 Project \textit{Alcuni problemi associati a funzionali integrali: riscoperta strumenti classici e nuovi sviluppi}, codice CUP\_E55\-F22\-000\-270\-001, and has received funding from the MIUR PRIN 2017 Project ``Gradient Flows, Optimal Transport and Metric Measure Structures''.
The second-named author is partially supported by the INdAM--GNAMPA 2022 Project \textit{Analisi geometrica in strutture subriemanniane}, codice CUP\_E55\-F22\-000\-270\-001, and has received funding from the European Research Council (ERC) under the European Union’s Horizon 2020 research and innovation program (grant agreement No.~945655).
}

\date{\today}



\begin{abstract}
Given $\alpha\in(0,1]$ and $p\in[1,+\infty]$, we define the space $\mathcal{DM}^{\alpha,p}(\mathbb R^n)$ of $L^p$ vector fields whose $\alpha$-divergence is a finite Radon measure, extending the theory of divergence-measure vector fields to the distributional fractional setting.  
Our main results concern the absolute continuity properties of the $\alpha$-divergence-measure with respect to the Hausdorff measure and fractional analogues of the Leibniz rule and the Gauss--Green formula.
The sharpness of our results is discussed via some explicit examples.
\end{abstract}

\maketitle


\section{Introduction}

\subsection{The classical framework}

The theory of \emph{divergence-measure fields} in the Euclidean space naturally emerged as the appropriate setting for the study of minimal regularity conditions allowing for integration-by-parts and Gauss--Green formulas.
Since Anzellotti's seminal paper~\cite{Anzellotti84}, several fundamental results have been established in the last 20 years, such as Leibniz rules for divergence-measure fields and suitably weakly differentiable scalar functions, well-posedness of generalized normal traces on rectifiable sets, and in\-te\-gra\-tion-by-parts formulas under minimal regularity assumptions, see~\cites{Ambrosio-Crippa-Maniglia05,Chen-et-al19,Chen-Frid99,Chen-Frid01,comi2022representation,Comi-Payne20,Crasta-DeCicco19-Anzellotti,Crasta-et-al22,CD5,Crasta-DeCicco19-An,Frid12,Silhavy09,Silhavy07,Silhavy19}.
Since its beginning, the theory of divergence-measure fields have found numerous applications in several areas, including Continuum Mechanics~\cites{Silhavy05,Degiovanni-et-al99,Schu}, hyperbolic systems of conservation laws~\cites{Ambrosio-Crippa-Maniglia05,Chen-Frid99,DeLellis-Otto-Westdickenberg03,ChTo}, gas dynamic~\cites{Chen-Frid03} and Dirichlet problems for the $1$-Laplacian operator and prescribed mean curvature-type equations \cite{Cicalese-Trombetti03,Kawhol-Schuricht07,leonardi2023prescribed,MR3813962,Scheven-Schmidt16,Leonardi-Saracco22,MR2502520,MR3939259,MR4241342,MR3978950}, just to name a few.
For recent extensions to non-Euclidean frameworks, we refer to \cite{brue2022constancy,ComiMagna,BuffaComiMira}.

The basic definition goes as follows (see \cref{subsec:notation} for the notation).
Given $p\in[1,+\infty]$, we say that a vector field $F\in L^p(\R^n;\R^n)$ has divergence-measure, and we write $F\in\DM^{1,p}(\R^n)$, if there exists a finite Radon measure $\divm F\in\M(\R^n)$ such that 
\begin{equation}
\label{eq:ibp_classical}
\int_{\R^n}F\cdot\nabla\xi\di x
=
-
\int_{\R^n}\xi\di\divm F
\end{equation} 
for all $\xi \in C^{\infty}_{c}(\R^n)$.
The integration-by-parts formula~\eqref{eq:ibp_classical} clearly generalizes the usual Divergence Theorem.
In fact, if the vector field~$F$ is sufficiently regular, say $F\in\Lip_{\loc}(\R^n;\R^n)$, then $\divm F=\div F\,\Leb{n}$ in~\eqref{eq:ibp_classical}, where $\Leb{n}$ is the $n$-dimensional measure.

As for the analogous case of functions with bounded variation, the two principal questions regarding $\DM^{1,p}$ vector fields concern the absolute continuity properties of the divergence-measure with respect to the Hausdorff measure $\Haus{s}$, for $s\in[0,n]$, and the well-posedness of a Leibniz rule with suitable scalar functions. 

The absolute continuity properties of the divergence-measure of a $\DM^{1,p}$ vector field with respect to the Hausdorff measure hold as follows, see~\cite{Silhavy05}*{Th.~3.2 and Exam.~3.3}.

\begin{theorem}[Absolute continuity properties of the divergence-measure]
\label{res:abs_div_classical}
Assume that $F \in \DM^{1, p}(\R^{n})$ with $p\in[1,+\infty]$. 
We have the following cases:
\begin{enumerate}[label=(\roman*),ref=\roman*,topsep=1ex,itemsep=1ex,leftmargin=6ex]
\item\label{item:classical_subcritical} 
if $p \in \left [1, \frac{n}{n - 1} \right )$, then $\divm\, F$ does not enjoy any absolute continuity property;

\item 
if $p \in \left [\frac{n}{n - 1}, + \infty \right )$, then $|\divm F|(B) = 0$ on Borel sets $B$ of $\sigma$-finite $\Haus{n - \frac{p}{p - 1}}$ measure;

\item if $p = + \infty$, then $|\divm F| \ll \Haus{n - 1}$. 
\end{enumerate}

\end{theorem}

The Leibniz rule involving $\DM^{1,p}$ vector fields and Sobolev functions is stated in \cref{res:leibniz_classical} below, for which we refer to~\cite{Chen-et-al19}*{Prop.~3.1}, \cite{Chen-Frid99}*{Th. 3.1}, \cite{ComiPhD}*{Th.~3.2.3} and~\cite{Frid14}*{Th.~2.1}.
Here and in the following, for $x\in\R^n$, we let 
\begin{equation}
\label{eq:precise_repres}
g^\star(x)
=
\begin{cases}
\displaystyle\lim_{r\to0^+}\mint{-}_{B_r(x)}g(y)\di y
&
\text{if the limit exists},
\\[2ex]
0
&
\text{otherwise},
\end{cases}
\end{equation}
be the \emph{precise representative} of $g\in L^1_{\loc}(\R^n)$.
For the notion of \emph{(Anzellotti's) pairing measure} briefly recalled in the statement, we refer the reader to~\cite{Anzellotti84}*{Def. 1.4}, \cite{Chen-Frid99}*{Th.~3.2}, or to \cite{Crasta-DeCicco19-Anzellotti}*{Sec. 2.5} for a more general formulation.

\begin{theorem}[Leibniz rule for $\DM^{1,p}$ vector fields and weakly differentiable functions]
\label{res:leibniz_classical}
Let $p,q\in[1,+\infty]$ be such that $\frac1p+\frac1q=1$. If $F \in \DM^{1, p}(\R^n)$ and 
\begin{equation*}
g \in \begin{cases} L^{\infty}(\R^n) \cap W^{1,q}(\R^n) & \text{ for } p \in [1, + \infty),\\[1ex]
L^{\infty}(\R^n) \cap BV(\R^n) & \text{ for } p = + \infty,
\end{cases}
\end{equation*}
then $gF \in \DM^{1,r}(\R^n)$ for all $r \in [1,p]$, with
\begin{equation} \label{eq:leibniz_rule_classical_p} 
\divm(gF) = g^{\star}\, \divm F + (F, Dg)_q 
\quad
\text{in}\ \M(\R^n). \end{equation}
Here\begin{equation*}
(F, Dg)_q = \begin{cases} F \cdot \nabla g \, \Leb{n} 
& 
\text{if}\ q>1,\ \text{or}\ q = 1\ \text{and}\ g \in L^\infty(\R^n) \cap W^{1,1}(\R^n), \\[1ex]
(F, Dg) 
& 
\text{if}\ q = 1\ \text{and}\ g \in L^\infty(\R^n) \cap ( BV(\R^n) \setminus W^{1,1}(\R^n)),
\end{cases}
\end{equation*}
is the \emph{pairing measure} between $F$ and~$Dg$, where $(F,Dg)$ is the unique weak limit 
\begin{equation*}
F\cdot\nabla(\rho_\eps*g)\,\Leb{n}
\weakto
(F,Dg)
\quad
\text{in}\ \M(\R^n)\
\text{as}\ \eps\to0^+,
\end{equation*}
being $\rho_\eps=\eps^{-n}\rho\left(\frac{\cdot}{\eps}\right)$ for $\eps>0$, with $\rho\in C^\infty_c(\R^n)$ any non-negative radially symmetric function such that $\supp\rho\subset B_1$ and $\int_{B_1}\rho\di x=1$. 

\end{theorem}

\begin{remark}[Choice of $g^\star$ in~\eqref{eq:leibniz_rule_classical_p} for $p<+\infty$]
For $p<+\infty$, the function $g^{\star}$ appearing in~\eqref{eq:leibniz_rule_classical_p} can be defined in a more specific way.
For $p \in \left [1, \frac{n}{n - 1} \right )$, $g^\star$ can be taken as the continuous representative of $g$.
Instead, for $p \in \left [ \frac{n}{n - 1}, + \infty \right )$, $g^\star$ can be taken as the \textit{$q$-quasicontinuous representative} of $g$.
See~\cite{ComiPhD}*{Sec.~3.2} for a more detailed discussion.
\end{remark}

\subsection{Fractional divergence-measure fields}

The aim of the present note is to introduce a fractional analogue of the theory of divergence-measure fields, following the distributional approach to fractional spaces recently introduced and studied by the authors and collaborators in the series of papers~\cite{Brue-et-al20,Comi-Stefani19,Comi-et-al21,Comi-Stefani22-L,Comi-Stefani22-A,Comi-Stefani23,Comi-Stefani23-On}.
For results close to the main topic of this paper, we also refer to~\cites{Silhavy22,Liu-Xiao22,Silhavy20}, even though our point of view is different.

In the fractional setting, for $\alpha\in(0,1)$, one has the integration-by-parts formula 
\begin{equation}
\label{eq:ibp_frac}
\int_{\R^n}F\cdot\nabla^\alpha\xi\di x
=
-
\int_{\R^n}\xi\,\div^\alpha F\di x
\end{equation} 
for all functions $\xi\in \Lip_c(\R^n)$ and vector fields $F\in \Lip_c(\R^n;\R^n)$, where
\begin{equation}
\label{eq:frac_nabla}
\nabla^\alpha \xi(x)
=
\mu_{n,\alpha}
\int_{\mathbb{R}^n}\frac{(\xi(y)-\xi(x))(y-x)}{|y-x|^{n+\alpha+1}}\di y,
\quad
x\in\R^n,
\end{equation}
is the \emph{fractional $\alpha$-gradient}, 
\begin{equation}
\label{eq:frac_div}
\div^\alpha F(x)
=
\mu_{n,\alpha}\int_{\R^n}\frac{(F(y)-F(x))\cdot(y-x)}{|y-x|^{n+\alpha+1}} \di y,
\quad
x\in\R^n,
\end{equation}
is the \emph{fractional $\alpha$-divergence}, and 
\begin{equation*}
\mu_{n, \alpha} 
= 
2^{\alpha}\, \pi^{- \frac{n}{2}}\, \frac{\Gamma\left ( \frac{n + \alpha + 1}{2} \right )}{\Gamma\left ( \frac{1 - \alpha}{2} \right )}
\end{equation*}
is a renormalization constant, see~\cite{Comi-Stefani19}*{Sec.~2.2} for a detailed exposition.
According to the main results of~\cites{Brue-et-al20,Comi-Stefani22-A}, with a slight (but justified) abuse of notation, we may identify~\eqref{eq:frac_nabla} with the usual gradient~$\nabla$ for $\alpha=1$, and with the vector-valued Riesz transform $\nabla^0=R$ for $\alpha=0$ (see \cref{subsec:notation} for the definition).

As already done by the authors in the case of scalar functions, the basic idea is now to use formula~\eqref{eq:ibp_frac} to define a fractional analogue of the divergence-measure~\eqref{eq:ibp_classical}.

\begin{definition}[$\DM^{\alpha,p}$ vector fields]
\label{def:dm_field}
Let $\alpha \in (0,1]$ and $p\in[1,+\infty]$. 
A vector field $F \in L^{p}(\R^n; \R^{n})$ has \emph{fractional $\alpha$-divergence-measure}, and we write $F \in \DM^{\alpha, p}(\R^n)$, if 
\begin{equation*}
\sup\set*{\int_{\R^n} F \cdot \nabla^{\alpha} \xi \di x : \xi\in C^\infty_c(\R^n),\ \|\xi\|_{L^\infty(\R^n)}\le1}<+\infty.
\end{equation*}
\end{definition}

The case $\alpha = 1$ in \cref{def:dm_field} corresponds to classical divergence-measure fields.
Without loss of generality, we always assume $n \ge 2$, since for $n=1$ one clearly identifies $\DM^{\alpha, p}(\R) = BV^{\alpha,p}(\R)$, the space of $L^p$ functions with finite totale fractional $\alpha$-variation, see the aforementioned~\cite{Comi-et-al21,Comi-Stefani22-L,Comi-Stefani23} for an extensive presentation of $BV^{\alpha,p}$ functions on $\R^n$. 
We also observe that $BV^{\alpha,p}(\R^n; \R^n) \subset \DM^{\alpha,p}(\R^n)$ for $n\ge2$, with strict inclusion at least in the case $p \in \left [1, \frac{n}{n-\alpha} \right )$, due to the fact that the vector fields in \cref{exa:F_frac} below cannot belong to $BV^{\alpha,p}(\R^n; \R^n)$, in the light of \cite{Comi-et-al21}*{Th.~1}.

Similarly to the case of $BV^{\alpha,p}$ functions (see \cite{Comi-Stefani19}*{Th.~3.2} and \cite{Brue-et-al20}*{Th.~5}), we can state the following structural result for $\DM^{\alpha, p}$ vector fields. 
The proof is very similar to the one of \cite{Comi-Stefani19}*{Th.~3.2} and is therefore omitted.

\begin{theorem}[Structure Theorem for $\DM^{\alpha, p}$ vector fields]\label{res:structure_frac_DM}
Let $\alpha \in (0,1]$ and $p \in [1, + \infty]$.
A vector field $F \in L^{p}(\R^{n};\R^n)$ belongs to $\DM^{\alpha,p}(\R^n)$ if and only if there exists a finite Radon measure $\divm^{\alpha} F \in \M(\R^n)$ such that 
\begin{equation}\label{eq:DM_alpha_p_duality} 
\int_{\R^n} F \cdot \nabla^{\alpha} \xi \di x= - \int_{\R^n} \xi \di \divm^{\alpha} F
\end{equation}
for all $\xi \in C^{\infty}_{c}(\R^n)$. 
In addition, for any open set $U\subset\R^n$, it holds
\begin{equation}\label{eq:frac_div-meas_tot}
|\divm^{\alpha} F|(U) = \sup\set*{\int_{\R^n} F \cdot \nabla^{\alpha} \xi \di x : \xi\in C^\infty_c(U),\ \|\xi\|_{L^\infty(U)}\le1}.
\end{equation}
\end{theorem}

If the vector field is sufficiently regular, say $F\in\Lip_c(\R^n;\R^n)$ for instance, then the fractional divergence-measure given by \cref{res:structure_frac_DM} is $\divm^\alpha F=\div^\alpha F\,\Leb{n}$, where $\div^\alpha F$ is as in~\eqref{eq:frac_div}.
Moreover, thanks to~\cref{res:structure_frac_DM}, the linear space 
\begin{equation*}
\DM^{\alpha,p}(\R^n)
=
\set*{F\in L^p(\R^n;\R^n) : |\divm^\alpha F|(\R^n)<+\infty}
\end{equation*} 
endowed with the norm
\begin{equation*}
\|F\|_{\DM^{\alpha,p}(\R^n)}
=
\|F\|_{L^p(\R^n;\,\R^n)}
+
|\divm^\alpha F|(\R^n),
\quad
F\in\DM^{\alpha,p}(\R^n),
\end{equation*}
is a Banach space, and the fractional divergence-measure in~\eqref{eq:frac_div-meas_tot} is lower semicontinuous with respect to the $L^p$ convergence.

\begin{remark}
[On the space $\DM^{0,p}$]
\label{rem:alpha_0}
Although not strictly necessary for the purposes of the present paper, let us briefly comment on the case $\alpha=0$ in \cref{def:dm_field}.
By exploiting \cite{Brue-et-al20}*{Lem.~26}, if $F \in \DM^{0,p}(\R^n)$ for some $p \in (1, + \infty)$, then  
\begin{equation*}
\divm^0 F = \div^0F \Leb{n} = (R \cdot F)\, \Leb{n}
\end{equation*} 
with $R\cdot F\in L^p(\R^n)$, where $R=\nabla^0$ the vector-value Riesz transform (see \cref{subsec:notation} for the definition). 
Therefore, for $p\in(1,+\infty)$, we can write
\begin{equation*}
\DM^{0,p}(\R^n) = \set*{F \in L^p(\R^n;\R^n) : \div^0F \in L^1(\R^n)}.
\end{equation*}
Hence, if $F\in\DM^{0,p}(\R^n)$ for some $p \in (1, + \infty)$, then $|\divm^0 F| \ll\Leb{n}$.
The limiting cases $p\in\set*{1,+\infty}$ seem more intricate and we leave them for future investigations. 
\end{remark}

\subsection{Main results}

Our first main result deals with the absolute continuity properties of $\DM^{\alpha,p}$ vector fields with respect to the Hausdorff measure, extending \cref{res:abs_div_classical}.
 
\begin{theorem}[Absolute continuity properties of the fractional divergence-measure]
\label{res:abs_frac_div}
Let $\alpha \in (0, 1)$, $p \in [1, + \infty]$ and assume that $F \in \DM^{\alpha, p}(\R^n)$. 
We have the following cases: 
\begin{enumerate}[label=(\roman*),ref=\roman*,itemsep=1ex,topsep=1ex,leftmargin=6ex]

\item
\label{item:abs_frac_div_subcritical} 

if $p \in \left [1, \frac{n}{n - \alpha} \right )$, then $\divm^\alpha F$ does not enjoy any absolute continuity property;

\item \label{item:abs_frac_div_intermediate} 

if $p \in \left [\frac{n}{n - \alpha}, \frac{n}{1 - \alpha} \right )$, then $|\divm^\alpha F|(B) = 0$ on Borel sets $B\subset\R^n$ with $\sigma$-finite $\Haus{n - \frac{p}{p - 1+(1 - \alpha)\frac pn}}$ measure;

\item \label{item:abs_frac_div_supercritical} if $p \in \left [\frac{n}{1 - \alpha}, + \infty \right ]$, then $|\divm^\alpha F| \ll \Haus{n - \alpha - \frac{n}{p}}$. 
\end{enumerate}
\end{theorem}

In particular, \cref{res:abs_frac_div} tells that, if $F \in \DM^{\alpha, \infty}(\R^{n})$, then $|\divm^{\alpha} F| \ll \Haus{n - \alpha}$, exactly as in \cref{res:abs_div_classical} for $p=+\infty$.
For $p<+\infty$, instead, the properties of the fractional divergence-measure are different from the corresponding ones in the classical setting.
Indeed, as for the fractional variation of $BV^{\alpha,p}$ functions (see \cite{Comi-et-al21}*{Th.~1} for the corresponding result), the threshold $p=\frac n{1-\alpha}$ imposes an interesting change of dimension of the Hausdorff measure.
This is quite customary in the distributional fractional framework, and is essentially due to the mapping properties of Riesz potential $I_{1-\alpha}$, see~\cite{Comi-Stefani19}*{Sec.~2.3}.

Our second main result concerns Leibniz rules for $\DM^{\alpha,p}$-fields and Besov functions, see~\cite{Comi-Stefani22-L}*{Th~1.1} for the corresponding result for $BV^{\alpha,p}$ functions.
We refer to \cref{subsec:notation} for the definitions of fractional Sobolev and Besov spaces.

\begin{theorem}[Leibniz rule for $\DM^{\alpha,p}$ vector fields with Besov functions] \label{res:leibniz_rules}
Let $\alpha\in(0,1)$ and let $p,q\in[1,+\infty]$ be such that $\frac1p+\frac1q=1$.
If $F \in \DM^{\alpha,p}(\R^n)$ and
\begin{equation}
\label{eqi:cases_BV_alpha_p}
g\in
\begin{cases}
B^\alpha_{q,1}(\R^n)
&
\text{for}\
p\in\left[1,\frac n{n-\alpha}\right),
\\[2mm]
L^\infty(\R^n)\cap B^\gamma_{q,1}(\R^n)\
\text{with}\ \gamma\in\left(\gamma_{n,q,\alpha},1\right)
&
\text{for}\
p\in\left[\frac n{n-\alpha},\frac n{1-\alpha}\right),
\\[2mm]
L^\infty(\R^n)\cap B^\beta_{q,1}(\R^n)\
\text{with}\ \beta\in\left(\beta_{n,q,\alpha},1\right)
&
\text{for}\
p\in\left[\frac n{1-\alpha},+\infty\right),
\\[2mm]
L^\infty(\R^n)\cap W^{\alpha,1}(\R^n)
&
\text{for}\
p=+\infty,
\end{cases}
\end{equation}
where
\begin{equation*}
\beta_{n,q,\alpha} 
= \frac{1}{q}\,
\left(\alpha+n-\frac nq\right) \quad 
\text{and}
\quad
\gamma_{n,q,\alpha} = \frac{n}{n + (1 -\alpha)q},
\end{equation*}
then $g F \in \DM^{\alpha,r}(\R^n)$ for all $r \in [1, p]$, with
\begin{equation*} 
\divm^{\alpha}(gF) 
= 
g^\star\,\divm^\alpha F
+ 
F\cdot
\nabla^{\alpha}g\,
\Leb{n} 
+ 
\div^{\alpha}_{\rm NL}(g, F)\,\Leb{n}
\quad 
\text{in}\ \M (\R^n),
\end{equation*}
where 
\begin{equation*}
\div^\alpha_{\rm NL}(g,F)
=
\mu_{n,\alpha}\int_{\R^n}\frac{(g(y)-g(x))(F(y)-F(x))\cdot (y-x)}{|y-x|^{n+\alpha+1}}\di y,
\quad
x\in\R^n,
\end{equation*}
is the \emph{non-local fractional divergence} of the couple $(g, F)$, and satisfies
\begin{equation*}
\|\div^\alpha_{\rm NL}(g,F)\|_{L^1}
\le 
\mu_{n,\alpha}
\,
[g]_{B^\alpha_{q,1}(\R^n)}
\,
\|F\|_{L^p(\R^n;\,\R^n)}.
\end{equation*}
In addition,
\begin{equation}
\label{eq:ping}
\divm^\alpha(gF)(\R^n)
=
\int_{\R^n}\div^\alpha_{\rm NL}(g, F)\di x=0,
\end{equation}
and 
\begin{equation}
\label{eq:pong}
\int_{\R^n}
F \cdot \nabla^\alpha g\di x
=
-
\int_{\R^n}
g^\star \di \divm^\alpha F.
\end{equation}
\end{theorem}

\cref{res:leibniz_rules}, besides providing an extention of \cref{res:leibniz_classical}, provides a Gauss--Green formula for $\DM^{\alpha,\infty}$ vector fields on $W^{\alpha,1}$ sets.
For the definitions of the \emph{fractional reduced boundary} $\redb^\alpha E$ and of the inner  \emph{fractional normal} $\nu_E^{\alpha}\colon\redb^\alpha E\to\mathbb S^{n-1}$ of a set $E\subset\R^n$, we refer the reader to~\cite{Comi-Stefani19}*{Def.~4.7}. 

\begin{corollary}[Generalized fractional Gauss--Green formula]
\label{res:ggf_frac}
Let $\alpha\in(0,1)$. 
If $F \in \DM^{\alpha, \infty}(\R^n)$ and $\chi_E\in W^{\alpha,1}(\R^n)$, then
\begin{equation*}
\int_{E^{1}}
\di \divm^\alpha F
=
-\int_{\redb^\alpha E}
F \cdot \nu^\alpha_E\,|\nabla^\alpha\chi_E|
\di x,
\end{equation*}
where
\begin{equation*}
E^{1}
=
\set*{x\in\R^n : \exists\lim_{r\to0^+}\frac{|E\cap B_r(x)|}{|B_r(x)|}=1}.
\end{equation*} 
\end{corollary}

\cref{res:ggf_frac} immediately follows from~\eqref{eq:pong} with $g = \chi_E$, since $\chi_E^\star = \chi_{E^1}$ $\Haus{n-\alpha}$-a.e.\ by \cite{Ponce-Spector20}*{Prop.~3.1}, and therefore $|\divm^\alpha F|$-a.e.\ thanks to point~\eqref{item:abs_frac_div_supercritical} of \cref{res:abs_frac_div}.

\cref{res:ggf_frac} provides the most general version known so far of the fractional Gauss--Green formula proved in~\cite{Comi-Stefani19}*{Th.~4.2}. Unfortunately, we do not know if the assumption $\chi_E\in W^{\alpha,1}(\R^n)$ can be replaced with the weaker one $\chi_E\in BV^{\alpha,1}(\R^n)$ in \cref{res:ggf_frac}.
In fact, as observed in \cite{Comi-et-al21}, we do not know whether the precise representative $g^\star$ defined in~\eqref{eq:precise_repres} of $g\in BV^{\alpha, \infty}(\R^n)$ is well defined up to $\Haus{n-\alpha}$-negligible sets.
We plan to tackle this and other strictly-connected challenging open questions in future works.

\subsection{Organization of the paper}
In \cref{sec:proofs}, we collect all the needed intermediate results to prove our main theorems.
In particular, \cref{subsec:relation_DM_with_DM_frac} and  \cref{subsec:decay_estimates}
contain the proofs of points~\eqref{item:abs_frac_div_intermediate} and~\eqref{item:abs_frac_div_supercritical} of \cref{res:abs_frac_div}, respectively.
The proof of \cref{res:leibniz_rules}, instead, can be found in~\cref{subsec:proof_leibinz}.
\cref{sec:examples} collects several examples.
In \cref{subsec:examples} we show point~\eqref{item:abs_frac_div_subcritical} of \cref{res:abs_frac_div}, while in \cref{subsec:sharp} we discuss the sharpness of the other two points~\eqref{item:abs_frac_div_intermediate} and~\eqref{item:abs_frac_div_supercritical} of \cref{res:abs_frac_div}.

\section{Proofs of the main results}
\label{sec:proofs}

In this section, we provide the proofs of our main results \cref{res:abs_frac_div} and \cref{res:leibniz_rules}.
The proof of \cref{res:abs_frac_div} is split across Sections~\ref{subsec:examples}, \ref{subsec:relation_DM_with_DM_frac} and~\ref{subsec:decay_estimates}, while the proof of \cref{res:leibniz_rules} is given in \cref{subsec:proof_leibinz}.

\subsection{General notation}
\label{subsec:notation}

We start with a brief description of the main notation used in this paper. In order to keep the exposition the most reader-friendly as possible, we retain the same notation adopted in our works~\cites{Brue-et-al20,Comi-Stefani19,Comi-et-al21,Comi-Stefani22-L,Comi-Stefani22-A,Comi-Stefani23,Comi-Stefani23-On}.

\subsubsection*{Lebesgue and Hausdorff measures}

We let $\Leb{n}$ and $\Haus{\alpha}$ be the $n$-dimensional Le\-be\-sgue measure and the $\alpha$-dimensional Hausdorff measure on $\R^n$, respectively, with $\alpha\in[0,n]$.
We denote by $B_r(x)$ the standard open Euclidean ball with center $x\in\R^n$ and radius $r>0$. 
We let $B_r=B_r(0)$. 
Recall that $\omega_{n} = |B_1|=\pi^{\frac{n}{2}}/\Gamma\left(\frac{n+2}{2}\right)$ and $\Haus{n-1}(\partial B_{1}) = n \omega_n$, where $\Gamma$ is Euler's Gamma function.

\subsubsection*{Regular maps}

Let $\Omega\subset\R^n$ be an open (non-empty) set.
For $k \in \N_{0} \cup \set{+ \infty}$ and $m \in \N$, we let $C^{k}_{c}(\Omega ; \R^{m})$ and $\Lip_c(\Omega; \R^{m})$ be the spaces of $C^{k}$-regular and, respectively, Lipschitz-regular, $m$-vector-valued functions defined on~$\R^n$ with compact support in the open set~$\Omega\subset\R^n$. 
Analogously, we let $C^{k}_{b}(\Omega ; \R^{m})$ and $\Lip_b(\Omega; \R^{m})$ be the spaces of $C^{k}$-regular and, respectively, Lipschitz-regular, $m$-vector-valued bounded functions defined on the open set $\Omega\subset\R^n$. 
In the case $k = 0$, we drop the superscript and simply write $C_{c}(\Omega ; \R^{m})$ and $C_{b}(\Omega ; \R^{m})$.
 
\subsubsection*{Radon measures}
For $m\in\N$, the total variation on~$\Omega$ of the $m$-vector-valued Radon measure $\mu$ is defined as
\begin{equation*}
|\mu|(\Omega)
=
\sup\set*{\int_\Omega\phi\cdot d\mu : \phi\in C^\infty_c(\Omega;\R^m),\ \|\phi\|_{L^\infty(\Omega;\,\R^m)}\le1}.
\end{equation*}
We thus let $\M(\Omega;\R^m)$ be the space of $m$-vector-valued Radon measure  with finite total variation on $\Omega$.
We say that $(\mu_k)_{k\in\N}\subset\M(\Omega;\R^m)$ \emph{weakly converges} to $\mu\in\M (\Omega;\R^m)$, and we write $\mu_k\weakto\mu$ in $\M (\Omega;\R^m)$ as $k\to+\infty$, if 
\begin{equation}\label{eq:def_weak_conv_meas}
\lim_{k\to+\infty}\int_\Omega\phi\cdot d\mu_k=\int_\Omega\phi\cdot d\mu
\end{equation} 
for all $\phi\in C_c(\Omega;\R^m)$. Note that we make a little abuse of terminology, since the limit in~\eqref{eq:def_weak_conv_meas} actually defines the \emph{weak*-convergence} in~$\M (\Omega;\R^m)$. 

\subsubsection*{Lebesgue, Sobolev and \texorpdfstring{$BV$}{BV} spaces}

For any exponent $p\in[1,+\infty]$, we let $L^p(\Omega;\R^m)$ be the space of $m$-vector-valued Lebesgue $p$-integrable functions on~$\Omega$.
We let
\begin{equation*}
W^{1,p}(\Omega;\R^m)
=
\set*{u\in L^p(\Omega;\R^m) : [u]_{W^{1,p}(\Omega;\,\R^m)}=\|\nabla u\|_{L^p(\Omega;\,\R^{nm})}<+\infty}
\end{equation*}
be the space of $m$-vector-valued Sobolev functions on~$\Omega$, see~\cite{Leoni17}*{Ch.~11}, and
\begin{equation*}
BV(\Omega;\R^m)
=
\set*{u\in L^1(\Omega;\R^m) : [u]_{BV(\Omega;\,\R^m)}=|Du|(\Omega)<+\infty}
\end{equation*}
be the space of $m$-vector-valued functions of bounded variation on~$\Omega$, see~\cite{AFP00}*{Ch.~3}. 

\subsubsection*{Fractional Sobolev spaces}

For $\alpha\in(0,1)$ and $p\in[1,+\infty)$, we let
\begin{equation*}
W^{\alpha,p}(\Omega;\R^m)
=\set*{u\in L^p(\Omega;\R^m) : [u]_{W^{\alpha,p}(\Omega;\,\R^m)}=\left(\int_\Omega\int_\Omega\frac{|u(x)-u(y)|^p}{|x-y|^{n+p\alpha}}\,dx\,dy\right)^{\frac{1}{p}}\!<+\infty}
\end{equation*}
be the space of $m$-vector-valued fractional Sobolev functions on~$\Omega$, see~\cite{DiNezza-et-al12}. 
For $\alpha\in(0,1)$ and $p=+\infty$, we simply let
\begin{equation*}
W^{\alpha,\infty}(\Omega;\R^m)=\set*{u\in L^\infty(\Omega;\R^m) : \sup_{x,y\in \Omega,\, x\neq y}\frac{|u(x)-u(y)|}{|x-y|^\alpha}<+\infty},
\end{equation*}
so that $W^{\alpha,\infty}(\Omega;\R^m)=C^{0,\alpha}_b(\Omega;\R^m)$, the space of $m$-vector-valued bounded $\alpha$-H\"older continuous functions on~$\Omega$.

\subsubsection*{Besov spaces}

For $\alpha\in(0,1)$ and $p,q\in[1,+\infty]$, we let
\begin{equation*}
B^\alpha_{p,q}(\R^n;\R^m)
=
\set*{
u\in L^p(\R^n;\R^m)
:
[u]_{B^\alpha_{p,q}(\R^n;\,\R^m)}
<
+\infty
}
\end{equation*}
be the space of $m$-vector-valued Besov functions on $\R^n$, see~\cite{Leoni17}*{Ch.~17}, where 
\begin{equation*}
[u]_{B^\alpha_{p,q}(\R^n;\,\R^m)}
=
\begin{cases}
\left(
\displaystyle\int_{\R^n}
\frac{\|u(\cdot+h)-u\|_{L^p(\R^n;\,\R^m)}^q}{|h|^{n+q\alpha}}
\,dh
\right)^{\frac1q}
&
\text{if}\
q\in[1,+\infty),
\\[10mm]
\sup\limits_{h\in\R^n\setminus\set*{0}}
\dfrac{\|u(\cdot+h)-u\|_{L^p(\R^n;\,\R^m)}}{|h|^{\alpha}}
&
\text{if}\
q=\infty.
\end{cases}
\end{equation*}

\subsubsection*{Shorthand for scalar function spaces}

In order to avoid heavy notation, if the elements of a function space $\mathcal F(\Omega;\R^m)$ are real-valued (i.e., $m=1$), then we will drop the target space and simply write~$\mathcal F(\Omega)$.

\subsubsection*{Riesz potential}

Given $\alpha\in(0,n)$, we let
\begin{equation}\label{eq:Riesz_potential_def} 
I_{\alpha} f(x) 
= 
2^{-\alpha} \pi^{- \frac{n}{2}} \frac{\Gamma\left(\frac{n-\alpha}2\right)}{\Gamma\left(\frac\alpha2\right)}
\int_{\R^{n}} \frac{f(y)}{|x - y|^{n - \alpha}} \, dy, 
\quad
x\in\R^n,
\end{equation}
be the Riesz potential of order $\alpha$ of $f\in C^\infty_c(\R^n;\R^m)$. We recall that, if $\alpha,\beta\in(0,n)$ satisfy $\alpha+\beta<n$, then we have the following \emph{semigroup property}
\begin{equation}\label{eq:Riesz_potential_semigroup}
I_{\alpha}(I_\beta f)=I_{\alpha+\beta}f
\end{equation}
for all $f\in C^\infty_c(\R^n;\R^m)$. In addition, if $1<p<q<+\infty$ satisfy 
$
\frac{1}{q}=\frac{1}{p}-\frac{\alpha}{n},	
$
then there exists a constant $C_{n,\alpha,p}>0$ such that the operator in~\eqref{eq:Riesz_potential_def} satisfies
\begin{equation}\label{eq:Riesz_potential_boundedness}
\|I_\alpha f\|_{L^q(\R^n;\,\R^m)}\le C_{n,\alpha,p}\|f\|_{L^p(\R^n;\,\R^m)}
\end{equation}
for all $f\in C^\infty_c(\R^n;\,\R^m)$. As a consequence, the operator in~\eqref{eq:Riesz_potential_def} extends to a linear continuous operator from $L^p(\R^n;\R^m)$ to $L^q(\R^n;\R^m)$, for which we retain the same notation. For a proof of~\eqref{eq:Riesz_potential_semigroup} and~\eqref{eq:Riesz_potential_boundedness}, see~\cite{Stein70}*{Ch.~V, Sec.~1} or~\cite{Grafakos14-M}*{Sec.~1.2.1}.

\subsubsection*{Riesz transform} We let
\begin{equation}\label{eq:def_Riesz_transform}
R f(x)
=
\pi^{-\frac{n+1}2}\,\Gamma\left(\tfrac{n+1}{2}\right)\,\lim_{\eps\to0^+}\int_{\set*{|y|>\eps}}\frac{y\,f(x+y)}{|y|^{n+1}}\,dy,
\quad
x\in\R^n,
\end{equation}
be the (vector-valued) \emph{Riesz transform} of a (sufficiently regular) function~$f$. 
We refer the reader to~\cite{Grafakos14-M}*{Sec.~2.1 and~2.4.4}, \cite{Stein70}*{Ch.~III, Sec.~1} and~\cite{Stein93}*{Ch.~III} for a more detailed exposition. 
We warn the reader that the definition in~\eqref{eq:def_Riesz_transform} agrees with the one in~\cites{Stein93} and differs from the one in~\cites{Grafakos14-M,Stein70} for a minus sign.
The Riesz transform~\eqref{eq:def_Riesz_transform} is a singular integral of convolution type, thus in particular it defines a continuous operator $R\colon L^p(\R^n)\to L^p(\R^n;\R^{n})$ for any given $p\in(1,+\infty)$, see~\cite{Grafakos14-C}*{Cor.~5.2.8}.
We also recall that its components $R_i$ satisfy
\begin{equation*}
\sum_{i=1}^nR_i^2=-\mathrm{Id}
\quad
\text{on}\ L^2(\R^n),
\end{equation*}
see~\cite{Grafakos14-C}*{Prop.~5.1.16}.

\subsection{Approximation by smooth vector fields}

Here and in the rest of the paper, we let $(\rho_\eps)\subset C^\infty_c(\R^n)$ be a family of standard mollifiers as in~\cite{Comi-Stefani19}*{Sec.~3.3}.
The following approximation result is the natural generalization to $\DM^{\alpha,p}$  vector fields of~\cite{Comi-et-al21}*{Th.~4}.
We leave its proof to the reader.

\begin{theorem}[Approximation by $C^\infty\cap \DM^{\alpha,p}$ fields] 
\label{res:approx_DM}
Let $\alpha \in (0,1]$ and  $p \in [1, +\infty]$.
Let $F \in \DM^{\alpha,p}(\R^n)$ and define
$F_{\eps}= F*\rho_{\eps}$ for all $\eps>0$.
Then $(F_{\eps})_{\eps>0}\subset \DM^{\alpha,p}(\R^n) \cap C^{\infty}(\R^n; \R^n)$ with
$\divm^{\alpha} F_{\eps} = (\rho_{\eps} \ast \divm^{\alpha} F) \Leb{n}$
for all $\eps>0$.
Moreover, we have:
\begin{enumerate}[label=(\roman*),ref=\roman*,topsep=1ex,itemsep=1ex]
\item 
if $p<+\infty$, then $F_{\eps} \to F$ in $L^p(\R^n; \R^n)$ as $\eps\to0^+$;
if $p=+\infty$, then $F_{\eps} \to F$ in $L^q_{\loc}(\R^n; \R^n)$ as $\eps\to0^+$ for all $q\in[1,+\infty)$;
\item 
$\divm^{\alpha} F_{\eps} \weakto \divm^{\alpha} F$ in $\M(\R^n)$
and 
$|\divm^{\alpha} F_{\eps}|(\R^n) \to |\divm^{\alpha} F|(\R^n)$
as $\eps\to0^+$. 
\end{enumerate}
\end{theorem}

\subsection{Integration-by-parts with Sobolev tests}

For future convenience, we note that the integration-by-parts formula~\eqref{eq:DM_alpha_p_duality} actually holds for a wider class of test functions. 
To this aim, let us recall the notion of \emph{non-local fractional gradient}
\begin{equation*}
\nabla^\alpha_{\rm NL} (f,g)(x)
=
\mu_{n,\alpha}
\int_{\R^n}\frac{(f(y)-f(x))(g(y)-g(x))(y-x)}{|y-x|^{n+\alpha+1}}\,dy,
\quad
x\in\R^n,	
\end{equation*}
of a couple of functions $f,g \in \Lip_c(\R^n)$.
The operator $\nabla^\alpha_{\rm NL}$ can be continuously extended to Lebesgue and Besov spaces, see~\cite{Comi-Stefani22-L}*{Cor.~2.7} for the precise statement.

\begin{proposition}[$W^{1,q}\cap C_b$-regular test] \label{res:sobolev_test}
Let $\alpha \in (0, 1)$ and let $p,q \in [1,+\infty]$ be such that $\frac1p+\frac1q=1$.
If $F \in \DM^{\alpha, p}(\R^n)$, then
\begin{equation}
\label{eq:sobolev_test}
\int_{\R^{n}} F \cdot \nabla^{\alpha} \xi \di x 
= 
- \int_{\R^{n}} \xi \di \divm^{\alpha} F
\end{equation}
for all $\xi\in W^{1,q}(\R^n)\cap C_b(\R^n)$, and for all $\xi\in BV(\R^n)\cap C_b(\R^n)$ if $q=1$.
\end{proposition}

\begin{proof} 
The proof is analogous to the one of~\cite{Comi-et-al21}*{Prop.~3}, so we only sketch it for the reader's convenience.
By a routine regularization-by-convolution argument, it is not restrictive to assume that  
$\xi\in W^{1,q}(\R^n)\cap\Lip_b(\R^n)\cap C^\infty(\R^n)$.
Letting $(\eta_R)_{R>0}\subset C^\infty_c(\R^n)$ be a family of cut-off functions as in~\cite{Comi-Stefani19}*{Sec.~3.3}, by~\cite{Comi-Stefani22-A}*{Lems.~2.3 and~2.4} we can write
\begin{equation}
\label{eq:R_sobolev_test}
\int_{\R^n} \eta_R\, F \cdot \nabla^\alpha\xi\di x
=
\int_{\R^n} F \cdot \nabla^\alpha(\eta_R\xi)\di x
-
\int_{\R^n} \xi \, F \cdot\nabla^\alpha\eta_R\di x
-
\int_{\R^n} F \cdot \nabla^\alpha_{\rm NL}(\eta_R,\xi)\di x
\end{equation} 
for all $R>0$.
Moreover, since $\xi\eta_R\in C^\infty_c(\R^n)$, we have
\begin{equation*}
\int_{\R^n} F \cdot \nabla^\alpha(\eta_R\xi)\di x
=
-
\int_{\R^n} \eta_R\xi\di  \divm^\alpha F
\end{equation*}
for all $R>0$.
Since 
\begin{equation*}
\lim_{R\to+\infty}
\int_{\R^n} \xi \, F \cdot\nabla^\alpha\eta_R\di x
=
\lim_{R\to+\infty}
\int_{\R^n} F \cdot \nabla^\alpha_{\rm NL}(\eta_R,\xi)\di x
=0,
\end{equation*}
the conclusion follows by passing to the limit as $R\to+\infty$ in~\eqref{eq:R_sobolev_test}.
\end{proof}

\subsection{Relation between \texorpdfstring{$\DM^{\alpha,p}$ and $\DM^{1,p}$}{fractional DM and classical DM vector fields}}
\label{subsec:relation_DM_with_DM_frac}

We now deal with point~\eqref{item:abs_frac_div_intermediate} of \cref{res:abs_frac_div}.
To this aim, we study the relationship between $\DM^{1,p}$ and $\DM^{\alpha,p}$ vector fields.

As one may expect, $\DM^{1,p}$ vector fields can be regarded as $\DM^{\alpha,p}$ vector fields, but only locally with respect to the divergence-measure.
For $\alpha\in(0,1)$ and $p\in[1,+\infty]$, we write $F\in\DM^{\alpha,p}_{\loc}(\R^n)$ if $F\in L^p(\R^n;\R^n)$ and, for any $U\subset\R^n$ bounded open set, 
\begin{equation*}
\sup\set*{\int_{\R^n}F\cdot\nabla^\alpha \xi\di x : \xi\in C^\infty_c(\R^n),\ \|\xi\|_{L^\infty(\R^n)}\le 1,\ \supp\xi\subset U}
<+\infty.
\end{equation*}
Consequently, the Radon measure $\divm^\alpha F\in\M_{\loc}(\R^n)$ given by~\eqref{eq:DM_alpha_p_duality} may be such that  $|\divm^\alpha F|(\R^n)=+\infty$.
This issue is quite normal, and essentially due to the properties of Riesz potential, in view of the representation $\nabla^\alpha=\nabla I_{1-\alpha}$, see~\cite{Comi-Stefani19}*{Sec.~2.3}.

\begin{lemma}[Inclusion]
\label{res:DM_into_DM_frac}
If $\alpha\in(0,1)$ and $p\in[1,+\infty]$, then $\DM^{1,p}(\R^n)\subset\DM^{\alpha,p}_{\loc}(\R^n)$. 
\end{lemma}

\begin{proof}
Let $F\in\DM^{1,p}(\R^n)$.
Given $\xi\in C^\infty_c(\R^n)$, since $I_{1-\alpha}\xi\in C^\infty_b(\R^n)$ with $\nabla^\alpha\xi=\nabla I_{1-\alpha}\xi\in L^{p'}(\R^n)$, we can write
\begin{equation*}
\int_{\R^n}F\cdot\nabla^\alpha \xi\di x
=
\int_{\R^n}F\cdot\nabla I_{1-\alpha}\xi\di x
=
-
\int_{\R^n}I_{1-\alpha}\xi\di \divm F.
\end{equation*}
Hence, for any bounded open set $U\supset\supp\xi$, by~\cite{Comi-Stefani19}*{Lem.~2.4} we can find a constant $C_{n,\alpha,U}>0$, depending only on $n$, $\alpha$ and $\mathrm{diam}(U)$, such that  
\begin{equation*}
\abs*{\,\int_{\R^n}F\cdot\nabla^\alpha \xi\di x
\,}
\le 
C_{n,\alpha,U}|\divm F|(\R^n)\|\xi\|_{L^\infty(\R^n)}.
\end{equation*} 
This implies that $F\in\DM^{\alpha,p}_{\loc}(\R^n)$, as desired.
\end{proof}

The inclusion given by \cref{res:DM_into_DM_frac} can be somewhat reversed, as done in \cref{res:DM_frac_to_DM} below.
Note that this result, besides providing analogues of~\cite{Comi-Stefani19}*{Lem.~3.28}, \cite{Comi-Stefani22-A}*{Lem.~3.7} and~\cite{Comi-et-al21}*{Prop.~4}, proves point~\eqref{item:abs_frac_div_intermediate} of \cref{res:abs_frac_div}

\begin{lemma}[Relation between $\DM^{\alpha,p}$ and $\DM^{1,p}$] \label{res:DM_frac_to_DM}
Let $\alpha \in (0,1)$, $p\in\left(1,\frac{n}{1-\alpha}\right)$ and $q = \frac{np}{n - (1-\alpha)p}$. 
If $F \in \DM^{\alpha, p}(\R^n)$, then $G = I_{1 - \alpha} F \in \DM^{1,q}(\R^n)$, with
\begin{equation*}
\|G\|_{L^q(\R^n; \,\R^n)}
\le 
c_{n,\alpha,p}\,\|F\|_{L^p(\R^n;\, \R^n)}
\quad\text{and}\quad
\divm\, G = \divm^\alpha F\ \text{in}\ \M(\R^{n}).
\end{equation*}
As a consequence, the operator $I_{1-\alpha}\colon \DM^{\alpha,p}(\R^n)\to \DM^{1,q}(\R^n)$ is continuous.
Moreover, for $p \in \left [ \frac{n}{n-\alpha} , \frac{n}{1-\alpha} \right )$, if $F \in \DM^{\alpha,p}(\R^{n})$ then $|\divm^\alpha F|(B) = 0$ on Borel sets $B\subset\R^n$ of $\sigma$-finite $\Haus{n - \frac{q}{q - 1}}$ measure.  
\end{lemma}

\begin{proof}
Let $p'=\frac{p}{p-1}$, $q'=\frac q{q-1}$ and note that $r = \frac{n p'}{n + (1 - \alpha)p'}\in\left(1,\frac{n}{1-\alpha}\right)$.
By the Hardy--Littlewood--Sobolev inequality, we immediately get that $G = I_{1 - \alpha} F \in L^q(\R^n;\R^n)$.
Moreover, given $\xi \in C^{\infty}_{c}(\R^{n})$, we clearly have $I_{1 - \alpha} |\nabla \xi| \in L^{q'}(\R^{n})$, because $|\nabla \xi|\in L^r(\R^n)$.
Hence, by Fubini Theorem, we can write
\begin{equation}
\label{eq:ibp_Riesz}
\int_{\R^{n}} F \cdot \nabla^{\alpha} \phi \di x = \int_{\R^{n}} F \cdot I_{1 - \alpha} \nabla \phi \di x = \int_{\R^{n}} G \cdot \nabla \phi \di x
\end{equation}
for all $\xi \in C^{\infty}_{c}(\R^{n})$, proving that $\divm^{\alpha} F = \divm\, G$ in $\M(\R^{n})$.
The remaining part of the statement easily follows from \cref{res:abs_div_classical} (also see~\cite{Silhavy05}*{Th.~3.2}).
\end{proof}

\subsection{Decay estimates}
\label{subsec:decay_estimates}

We now deal with point~\eqref{item:abs_frac_div_supercritical} of \cref{res:abs_frac_div}.
To this aim, we prove some decay estimates of the fractional divergence-measure on balls.

Let us begin with the following result, which may be considered as a toy case for the more general result in \cref{res:decay_estimate_p} below. 

\begin{lemma}[Decay estimate for $\divm^\alpha F\ge0$]
\label{res:decay_estimate_pos}
Let $\alpha \in (0, 1]$ and $p \in [1, + \infty]$. 
If $F \in \DM^{\alpha, p}(\R^n)$ satisfies $\divm^{\alpha} F \ge 0$ on some open set $A\subset\R^n$, then 
\begin{equation} 
\label{eq:decay_estimate_pos}
\divm^{\alpha}F(B_{r}(x)) \le C_{n, \alpha, p} \, \|F\|_{L^{p}(\R^n; \,\R^{n})}\, r^{n - \alpha - \frac{n}{p}}. \end{equation}
for all $x\in A$ and $r>0$ such that  $B_{2r}(x)\subset A$.
\end{lemma}

\begin{proof}
Let $\xi \in C^{\infty}_{c}(B_2)$ be such that $\xi \ge 0$ and $\xi \equiv 1$ on $B_1$.
Then, for $x\in A$ and $r>0$ such that $B_{2r}(x)\subset A$, we can estimate
\begin{equation*} \divm^{\alpha}F(B_{r}(x)) \le \int_{\R^n} \xi\left (\frac{y - x}{r} \right ) \di \divm^{\alpha}F(y) = - \int_{\R^n} F(y) \cdot ( \nabla^{\alpha} \xi) \left (\frac{y - x}{r} \right ) r^{- \alpha} \di y. \end{equation*}
Thus we easily get
\begin{align*} \divm^{\alpha}F(B_{r}(x)) & \le \|F\|_{L^{p}(\R^n; \R^{n})}\, r^{- \alpha}\left ( \int_{\R^n} |\nabla^{\alpha} \xi(y)|^{p'} r^{n} \di y \right )^{\frac{1}{p'}}  \\
& = \|F\|_{L^{p}(\R^n; \R^{n})}\, \|\nabla^{\alpha} \xi \|_{L^{p'}(\R^n;\, \R^{n})}\, r^{n - \alpha - \frac{n}{p}}, 
\end{align*}
from which the conclusion immediately follows.
\end{proof}

\cref{res:decay_estimate_pos}, despite its simplicity, allows to recover the following rigidity result, which may be seen as the natural fractional analogue of~\cite{Phuc-Torres08}*{Th.~3.1}.

\begin{proposition}[Rigidity]
Let $\alpha \in (0, 1]$ and $p \in \left [1, \frac{n}{n-\alpha} \right ]$. 
If $F \in \DM^{\alpha, p}(\R^n)$ satisfies $\divm^{\alpha} F \ge 0$, then $\divm^\alpha F = 0$.
\end{proposition}

\begin{proof}
If $p<\frac{n}{n-\alpha}$, so that $n - \alpha - \frac{n}{p}<0$, then
\begin{equation*} 0 \le \divm^{\alpha}F(B_{r}) \le C_{n, \alpha, p}  \|F\|_{L^{p}(\R^n; \R^{n})} r^{n - \alpha - \frac{n}{p}} \end{equation*}
for all $r>0$ by \cref{res:decay_estimate_pos} in the case $x=0$.
Hence the conclusion follows by taking the limit as $r \to + \infty$.
If instead $p = \frac{n}{n-\alpha}$, then $I_\alpha \divm^\alpha F=\div^0 F$ in  $L^{\frac n{n-\alpha}}(\R^n)$, since
\begin{equation*}
\int_{\R^n}I_\alpha\xi \di\divm^\alpha F
=
-
\int_{\R^n} F\cdot\nabla^\alpha I_\alpha\xi\di x
=
- \int_{\R^n}F\cdot \nabla^0\xi\di x = \int_{\R^n} \xi \div^0 F \, \di x
\end{equation*}
for all $\xi\in C^\infty_c(\R^n)$ by \cref{res:sobolev_test}, \cref{rem:alpha_0} and \cite{Brue-et-al20}*{Prop. 7 and Lem.~26}.
However, for all $R > 0$ and $x \in \R^n$ we also have 
\begin{equation*}
I_\alpha \divm^\alpha F(x) 
\ge 
c_{n, \alpha} 
\int_{B_R} \frac{1}{|x-y|^{n-\alpha}} \di \divm^{\alpha} F(y) 
\ge 
\tilde{c}_{n, \alpha}\, 
\frac{\divm^{\alpha} F(B_R)}{(|x| + R)^{n-\alpha}}, 
\end{equation*}
and thus $I_\alpha \divm^\alpha F\notin L^{\frac n{n-\alpha}}$  unless $\divm^\alpha F=0$.
The proof is complete.
\end{proof}

To remove the non-negativity assumption $\divm^\alpha F\ge0$ from the conclusion~\eqref{eq:decay_estimate_pos} in  \cref{res:decay_estimate_pos} we need to deal with integration-by-parts for $\DM^{\alpha,p}$ fields on balls. 
The following result is the analogue of~\cite{Comi-et-al21}*{Th.~9}.

\begin{theorem}[Integration by parts on balls] \label{res:int_by_parts_E_ball} 
Let $\alpha \in (0, 1)$ and 
$p\in\left(\frac{1}{1-\alpha},+\infty\right]$. 
If $F \in \DM^{\alpha, p}(\R^{n})$, 
$\xi\in\Lip_c(\R^{n})$
and $x \in \R^n$, then 
\begin{equation} 
\label{eq:int_by_parts_on_ball} 
\int_{B_r(x)} F \cdot \nabla^{\alpha} \xi \di y 
+ 
\int_{\R^{n}} \xi F \cdot \nabla^{\alpha} \chi_{B_r(x)}\di y  
+ 
\int_{\R^{n}} F \cdot \nabla^{\alpha}_{\rm NL} (\chi_{B_r(x)}, \xi) \di y
=
-\int_{B_r(x)} \xi \di \divm^{\alpha} F
\end{equation}
for $\Leb{1}$-a.e.\ $r > 0$.
\end{theorem}

\begin{proof}
The proof is very similar to that of~\cite{Comi-et-al21}*{Th.~9}, so we only sketch it for the reader's convenience.
Fix $x \in \R^{n}$ and $\xi \in \Lip_{c}(\R^{n})$ be fixed.

In the case $p=+\infty$, we consider $h_{\eps, r, x}\in\Lip_c(\R^n)$ for $\eps>0$ and $r>0$ defined as
\begin{equation*} 
h_{\eps, r, x}(y) 
= 
\begin{cases} 
1 
& \text{if} \ 0 \le |y - x| \le r, \\[3mm]
\dfrac{r + \eps - |y - x|}{\eps} 
& \text{if} \ r < |y - x| < r + \eps, \\[4mm] 
0 
& \text{if} \  |y -x| \ge r + \eps, 
\end{cases} 
\end{equation*}
for all $y\in\R^n$.
By~\cite{Comi-Stefani19}*{Lem.~5.1}, 
$\nabla^\alpha h_{\eps,r,x}\in L^1(\R^n;\R^n)$ with
\begin{equation} \label{eq:frac_grad_cutoff_ball} 
\nabla^{\alpha} h_{\eps, r, x}(y) 
= 
\frac{\mu_{n, \alpha}}{\eps (n + \alpha - 1)} 
\int_{B_{r + \eps}(x) \setminus B_r(x)} 
\frac{x - z}{|x - z|}\,|z - y|^{1 - n - \alpha}\di z 
\end{equation}
for $\Leb{n}$-a.e.\ $y \in \R^{n}$. 

Since $h_{\eps, r, x}(y) \to \chi_{\closure{B_r(x)}}(y)$ as $\eps \to 0^+$ for all $y \in \R^{n}$ and $|\divm^{\alpha} F|(\de B_r(x)) = 0$ for $\Leb{1}$-a.e.\ $r > 0$, we can use $h_{\eps, r, x}$ to approximate $\chi_{B_r(x)}$ in~\eqref{eq:int_by_parts_on_ball}.
On the one hand, since $h_{\eps, r, x}\,\phi\in\Lip_c(\R^n;\R^n)$, by \cref{res:sobolev_test} we have 
\begin{equation} \label{eq:IBP_E_1} 
\int_{\R^{n}} F \cdot \nabla^{\alpha}(h_{\eps, r, x}\,\phi) \di y 
= 
- \int_{\R^{n}} h_{\eps, r, x}\,\phi \di \divm^{\alpha} F.
\end{equation}
On the other hand, by~\cite{Comi-Stefani19}*{Lem.~2.6}, we can compute
\begin{equation} 
\label{eq:Leibniz_rule_h_phi} 
\nabla^{\alpha}(h_{\eps, r, x}\,\phi) 
= 
h_{\eps, r, x}\, \nabla^{\alpha}\phi 
+ 
\phi \, \nabla^{\alpha} h_{\eps, r, x} 
+ 
\nabla_{\rm NL}^{\alpha}(h_{\eps, r, x}, \phi). 
\end{equation}
One then has to deal with each term of the right-hand side of~\eqref{eq:Leibniz_rule_h_phi} separately. 
The most difficult term is the second one, for which one has to observe that, by~\eqref{eq:frac_grad_cutoff_ball}, 
\begin{align*} 
\int_{\R^{n}} \xi(y) \, F(y)
&\cdot 
\nabla^{\alpha} h_{\eps, r, x}(y) \di y
\\ 
&= 
\frac{\mu_{n, \alpha}}{\eps (n + \alpha - 1)} 
\int_{\R^{n}} \xi(y) \, F(y)
\cdot 
\int_{B_{r+\eps}(x) \setminus B_r(x)} \frac{x - z}{|x - z|} |z - y|^{1 - n - \alpha}\di z \di y
\\
& = 
\int_{B_{r+\eps}(x) \setminus B_r(x)}\frac{x - z}{|x - z|}\cdot \int_{\R^{n}} F(y) \, \xi(y) \,|z - y|^{1 - n - \alpha} \di y \di z 
\\
& = 
\int_{r}^{r + \eps} \int_{\partial B_{\rho}(x)} \frac{x - z}{|x - z|}\cdot \int_{\R^{n}} F(y) \, \xi(y)\,|z - y|^{1 - n - \alpha} \di y \di  \Haus{n - 1}(z) \di  \rho.
\end{align*}
Hence, by Lebesgue's Differentiation Theorem,
\begin{equation*}
\begin{split}
\lim_{\eps\to0}\,
\frac{1}{\eps}
\int_{\R^{n}} 
& 
\xi(y) \, F(y) \cdot \int_{B_{r+\eps}(x) \setminus B_r(x)} \frac{x - z}{|x - z|}\, |z - y|^{1 - n - \alpha}\di z \di y
\\
&=
\int_{\partial B_r(x)} \frac{x - z}{|x - z|}\cdot \int_{\R^{n}} F(y) \, \xi(y)\,|z - y|^{1 - n - \alpha} \di y \di  \Haus{n - 1}(z)
\\
&=
\int_{\R^{n}} \xi(y) \, F(y) \cdot\int_{\R^n} |z - y|^{1 - n - \alpha}\di D\chi_{B_r(x)}(z)\di y
\end{split}
\end{equation*}
for $\Leb{1}$-a.e.\ $r > 0$. 
Thus, by~\cite{Comi-Stefani19}*{Th.~3.18, Eq.~(3.26)}, we get that 
\begin{equation} \label{eq:IBP_conv_3} 
\begin{split}
\lim_{\eps\to0}
\int_{\R^{n}} 
& 
\xi \, F \cdot \nabla^{\alpha}  h_{\eps, r, x} \di y 
\\
& = 
\frac{\mu_{n, \alpha}}{n + \alpha - 1} 
\int_{\R^{n}} \xi(y) \, F(y) \cdot\int_{\R^{n}} |z - y|^{1 - n - \alpha} \di  D \chi_{B_{r}(x)}(z)  \di y 
\\
& = 
\int_{\R^{n}} \xi \, F \cdot \nabla^{\alpha} \chi_{B_{r}(x)} \di y  
\end{split}
\end{equation}
for $\Leb{1}$-a.e.\ $r > 0$.
The other terms are easier and hence left to the reader.

In the case $p \in \left (\frac{1}{1 - \alpha}, +\infty \right )$, instead, one regularizes  $F\in\DM^{\alpha,p}(\R^n)$ to $(F_{\eps})_{\eps>0} 
\subset 
\DM^{\alpha, p}(\R^{n})
\cap 
L^{\infty}(\R^{n};\R^n) 
\cap 
C^{\infty}(\R^{n}; \R^n)$  via convolution to reduce to the previous case $p=+\infty$.
The conclusion then follows by exploiting the convergence properties given by \cref{res:approx_DM} and recalling that, thanks to \cite{Comi-et-al21}*{Cor.~1}, $\nabla^{\alpha} \chi_{B_{r}(x)} \in L^{q}(\R^{n}; \R^{n})$ for any $p \in \left (\frac{1}{1 - \alpha}, \infty \right )$, where $q = \frac{p}{p-1}$, and that $\nabla^\alpha_{\rm NL}(\chi_{B_r(x)}, \xi) \in L^q(\R^n; \R^n)$ as well, thanks to \cite{Comi-Stefani22-L}*{Cor.~2.7}.
We leave the details to the reader.
\end{proof}

We are now ready to generalize \cref{res:decay_estimate_pos} beyond the non-negativity assumption, as done in~\cite{Comi-et-al21}*{Th.~10} for $BV^{\alpha,p}$ functions.

\begin{theorem}[Decay estimates for $\DM^{\alpha,p}$ functions for $p>\frac 1{1-\alpha}$]
\label{res:decay_estimate_p} 
Let $\alpha \in (0,1)$ and
$p\in\left(\frac{1}{1-\alpha},+\infty\right]$.
There exist two constants $A_{n,\alpha,p}, B_{n,\alpha,p} > 0$, depending on $n$, $\alpha$ and $p$ only, with the following property. 
If $F \in \DM^{\alpha,p}(\R^n)$ then,
for $|\divm^{\alpha} F|$-a.e.\ $x \in \R^n$, there exists $r_x > 0$ such that
\begin{equation}
\label{eq:decay_div_alpha_F_B_1_prime} 
|\divm^{\alpha} F|(B_r(x)) 
\le 
A_{n,\alpha,p} 
\|F\|_{L^p(\R^n; \R^n)}\,
r^{\frac{n}{q} - \alpha} 
\end{equation}
and
\begin{equation}
\label{eq:decay_div_alpha_F_B_2_prime}  
|\divm^{\alpha} (\chi_{B_r(x)} F)|(\R^{n}) 
\le 
B_{n,\alpha,p} 
\|F\|_{L^p(\R^n; \R^n)}\,
r^{\frac{n}{q} - \alpha} 
\end{equation}
for all $r \in (0, r_x)$, 
where $q\in[1,+\infty)$ is such that $\frac1p+\frac1q=1$.
\end{theorem}

\begin{proof}
The proof follows the same line of that of~\cite{Comi-et-al21}*{Th.~10}, so we only sketch it for the reader's ease.
Since $F \in \DM^{\alpha,p}(\R^n)$, by the Polar Decomposition Theorem for Radon measures there exists a Borel function $\sigma_{F}^{\alpha}\colon\R^n\to\R$ such that 
\begin{equation}
\label{eq:polar_decomp_div_alpha}
\divm^{\alpha} F 
= 
\sigma_{F}^{\alpha}\,|\divm^{\alpha} F| 
\quad\text{with}\quad 
|\sigma_{F}^{\alpha}(x)| = 1\ 
\text{for} \ 
|\divm^{\alpha} F|\text{-a.e.}\ 
x \in \R^{n}.
\end{equation}
For $x\in\R^n$ such that $|\sigma_{F}^{\alpha}(x)| = 1$, given $r>0$ we define $\xi_{x,r}\colon\R^n\to\R$ as
\begin{equation} 
\label{eq:phi_x_r_def}
\xi_{x, r}(y) = 
\begin{cases} 
\sigma_{F}^{\alpha}(x) 
& \text{if} \,  y \in B_{r}(x),\\
\sigma_{F}^{\alpha}(x) \left (2 - \frac{|y - x|}{r} \right ) 
& \text{if} \, y \in B_{2r}(x) \setminus B_{r}(x),\\
0 
& \text{if} \, y \notin B_{2 r}(x),
\end{cases}
\end{equation}
for all $y\in\R^n$.
Since 
$\xi_{x, r}\in \Lip_{c}(\R^{n})$ 
with
$\|\phi\|_{L^{\infty}(\R^{n})} \le 1$, we can find $r_x\in(0,1)$ such that
\begin{equation} 
\label{eq:decay_estimate_1_p} 
\int_{B_r(x)} \xi_{x, r}(y) \di \divm^{\alpha} F(y) 
= 
\int_{B_{r}(x)} \sigma_{F}^{\alpha}(x) \, \sigma_{F}^{\alpha}(y) \di |\divm^{\alpha}F|(y) 
\ge 
\frac{1}{2} |\divm^{\alpha} F|(B_r(x)) 
\end{equation}
for all $r\in(0,r_x)$. 
Also, by~\eqref{eq:int_by_parts_on_ball}, we can estimate
\begin{equation}
\label{eq:decay_estimate_1.5_p}
\begin{split} 
\int_{B_r(x)} \xi_{x, r} \di \divm^{\alpha} F 
&\le 
\abs*{\,\int_{B_{r}(x)} F \cdot \nabla^{\alpha} \xi_{x, r} \di y\,} 
+
\abs*{\,\int_{\R^{n}} \xi_{x, r} \, F \cdot \nabla^{\alpha} \chi_{B_r(x)} \di x\,}\\
&\quad+
\abs*{\, \int_{\R^{n}} F \cdot \nabla^{\alpha}_{\rm NL} (\chi_{B_r(x)}, \xi_{x, r}) \di y \, } 
\end{split}
\end{equation}
for $\Leb{1}$-a.e.\ $r\in(0,r_x)$.
Hence the inequality in~\eqref{eq:decay_div_alpha_F_B_1_prime} follows by estimating the three terms in the right-hand side of~\eqref{eq:decay_estimate_1.5_p}, recalling~the scaling property of $\nabla^\alpha$, \cite{Comi-et-al21}*{Cor.~1} and \cite{Comi-Stefani22-L}*{Cor.~2.7}.
For the inequality in~\eqref{eq:decay_div_alpha_F_B_2_prime}, instead, one notes that, given any $\xi \in \Lip_{c}(\R^{n})$ with $\|\xi\|_{L^{\infty}(\R^{n})} \le 1$, from~\eqref{eq:int_by_parts_on_ball} it holds
\begin{align*} 
\abs*{\,\int_{B_r(x)} F \cdot \nabla^{\alpha} \xi \di y\,} 
&\le 
|\divm^{\alpha} F|(B_r(x)) 
+ 
\|F\|_{L^{p}(\R^{n};\,\R^n)} \, 
\|\nabla^{\alpha} \chi_{B_r(x)}\|_{L^{q}(\R^n; \R^{n})}\\
&\quad+ 
\|F\|_{L^{p}(\R^{n};\,\R^n)} \|\nabla^{\alpha}_{\rm NL}(\chi_{B_r(x)}, \xi)\|_{L^{q}(\R^{n}; \R^n)}
\end{align*}
for $\Leb{1}$-a.e.\ $r\in(0,r_x)$. The conclusion thus follows from~\eqref{eq:decay_div_alpha_F_B_1_prime} and again~\cite{Comi-et-al21}*{Cor.~1} and \cite{Comi-Stefani22-L}*{Cor.~2.7}.
We leave the details to the reader.
\end{proof}

As a consequence of \cref{res:decay_estimate_p}, we get the following result, in particular proving the validity of point~\eqref{item:abs_frac_div_supercritical} in \cref{res:abs_frac_div}.
Note that \cref{res:abs_continuity_p} below is actually relevant only in the case of point \eqref{item:abs_frac_div_supercritical} of \cref{res:abs_frac_div}, since 
$
n - \frac{p}{p-1+(1 - \alpha)\frac pn} \le n - \alpha - \frac{n}{p}
$
if and only if $p \ge \frac{n}{1 - \alpha}$ and $p \le \frac{n}{n - \alpha}$, but in this second case both exponents are negative.

\begin{corollary}[$|\divm^{\alpha} F|\ll \Haus{n -\alpha - \frac{n}{p}}$ for $p>\frac1{1-\alpha}$] 
\label{res:abs_continuity_p} 
Let $\alpha \in (0,1)$ and
$p\in\left(\frac{1}{1-\alpha},+\infty\right]$.
If $F \in \DM^{\alpha,p}(\R^n)$, then there exists a $|\divm^\alpha F|$-negligible set $Z_F^{\alpha,p}\subset\R^n$ such that
\begin{equation}
\label{eq:abs_cont_estimate}
|\divm^{\alpha} F| 
\le 
2^{\frac nq -\alpha}\,
\frac{A_{n,\alpha,p}}{\omega_{\frac nq-\alpha}} \,
\|F\|_{L^p(\R^{n}; \R^n)} \, 
\Haus{\frac nq - \alpha} \res \R^{n} \setminus Z_F^{\alpha,p},
\end{equation}
where $A_{n,\alpha,p}$ is as in~\eqref{eq:decay_div_alpha_F_B_1_prime} and $q\in[1,+\infty)$ is such that $\frac1p+\frac1q=1$.
\end{corollary}

\begin{proof}
By \cref{res:decay_estimate_p}, there exists a set $Z_{F}^{\alpha,p}\subset\R^n$, which we can assume to be Borel without loss of generality, such that $|\divm^{\alpha}F|(Z_{F}^{\alpha,p}) = 0$ and~\eqref{eq:decay_div_alpha_F_B_1_prime} holds for any $x \notin Z_{F}^{\alpha,p}$. 
Hence, for all 
$x\in\R^{n}\setminus Z_{F}^{\alpha,p}$, 
we have
\begin{equation*} 
\Theta^{*}_{\frac nq - \alpha}(|\divm^{\alpha} F|, x) 
= 
\limsup_{r \to 0^+} \frac{|\divm^{\alpha} F|(B_{r}(x))}{\omega_{\frac nq - \alpha} r^{\frac nq - \alpha}} 
\le 
\frac{A_{n, \alpha,p}}{\omega_{\frac nq - \alpha}}\, 
\|F\|_{L^{p}(\R^{n}; \R^n)}.
\end{equation*}
Inequality~\eqref{eq:abs_cont_estimate} thus follows from~\cite{AFP00}*{Th.~2.56}.
\end{proof}

\begin{remark}\cref{res:abs_continuity_p} holds true also in the limit case as $\alpha \to 1^-$. 
Indeed, if $F \in \DM^{1, \infty}(\R^n)$, then~\cite{Silhavy19}*{Prop.~1} implies that
\begin{equation*}
|\divm F| \le c_{n} \|F\|_{L^{\infty}(\R^n; \,\R^n)} \Haus{n-1} \res (\R^n \setminus Z^{1, \infty}_{F}),
\end{equation*}
for some constant $c_n > 0$ and any $|\divm F|$-negligible set $Z^{1, \infty}_{F}\subset\R^n$.
\end{remark}

\subsection{Proof of \texorpdfstring{\cref{res:leibniz_rules}}{fractional Leibniz rules}}
\label{subsec:proof_leibinz}

We begin with the following technical result.

\begin{lemma}[Zero total divergence-measure]
\label{res:zero_total_mesure}
Let $\alpha\in (0,1]$ and $p \in \left [1, \frac{n}{n-\alpha} \right )$. 
If $F \in \DM^{\alpha,p}(\R^n)$, then $\divm^\alpha F(\R^n)=0$.
\end{lemma}

\begin{proof}
Let $\eta \in C^{\infty}_c(B_2)$ be such that $\eta \equiv 1$ on $B_1$ and set $\eta_k(x) = \eta \left (\frac{x}{k} \right )$ for $k\in\N$ and $x\in\R^n$. 
By~\eqref{eq:DM_alpha_p_duality} and the $\alpha$-homogeneity of the fractional gradient, we have
\begin{align*}
\left | \int_{\R^n} \eta_k \di \divm^\alpha F \right |
&= 
\left | \int_{\R^n} F \cdot \nabla^\alpha \eta_k \di x \right | 
\\
&\le 
k^{\frac{n}{q} - \alpha}
\,
\|F\|_{L^p(\R^n; \,\R^n)} 
\, 
\|\nabla^\alpha \eta \|_{L^q(\R^n; \R^n)} \to 0 \quad
\text{as}\ k \to + \infty
\end{align*}
for $q > \frac{n}{\alpha}$, which means $p < \frac{n}{n - \alpha}$. 
Hence, by the Dominated Convergence Theorem with respect to the measure $|\divm^\alpha F|$, we get that
\begin{equation*}
\divm^\alpha F(\R^n) 
= 
\int_{\R^n} \di \divm^{\alpha} F 
= 
\lim_{k \to + \infty} \int_{\R^n} \eta_k \di \divm^\alpha F
=
0
\end{equation*}
concluding the proof.
\end{proof}

We can now deal with the Leibniz rule for $\DM^{\alpha,p}$ vector fields and bounded continuous Besov functions, in analogy with~\cite{Comi-Stefani22-L}*{Th.~3.1}. 
To this purpose, we need to recall the notion of {\em non-local fractional divergence}
\begin{equation*}
\div^\alpha_{\rm NL} (g, F)(x)
=
\mu_{n,\alpha}
\int_{\R^n}\frac{(g(y)-g(x))(F(y)-F(x)) \cdot (y-x)}{|y-x|^{n+\alpha+1}}\,dy,
\quad
x\in\R^n,	
\end{equation*}
of a couple $(g, F)$, where $g \in \Lip_c(\R^n)$ and $F \in \Lip_c(\R^n; \R^n)$.
The operator $\div^\alpha_{\rm NL}$ can be continuously extended to Lebesgue and Besov spaces, see~\cite{Comi-Stefani22-L}*{Cor.~2.7}.

\begin{theorem}[Leibniz rule for $\DM^{\alpha,p}$ with $C_b\cap B^\alpha_{q,1}$ for $\frac1p+\frac1q=1$]
\label{res:BV_alpha_p_leibniz_C_b}
Let $\alpha\in(0,1)$ and let $p,q\in[1,+\infty]$ be such that $\frac1p+\frac1q=1$.
If $F \in \DM^{\alpha, p}(\R^n)$ and $g\in C_b(\R^n)\cap B^\alpha_{q,1}(\R^n)$, then
$gF \in \DM^{\alpha,r}(\R^n)$ for all $r\in[1,p]$, with $\div^{\alpha}_{\rm NL}(g, F) \in L^1(\R^n)$ and
\begin{equation} 
\label{eq:BV_alpha_p_leibniz_C_b}
\divm^{\alpha}(gF) 
= 
g\, \divm^{\alpha} F 
+ 
F\cdot
\nabla^{\alpha}g\,
\Leb{n} 
+ 
\div^{\alpha}_{\rm NL}(g, F)\,\Leb{n}
\quad 
\text{in}\ \M (\R^n).
\end{equation}
In addition,
\begin{equation}
\label{eq:BV_alpha_p_leibniz_C_b_D_alpha_zero}
\divm^\alpha(gF)(\R^n)=0,
\qquad
\int_{\R^n}\div^\alpha_{\rm NL}(g,F)\di x=0,
\end{equation}
and 
\begin{equation}
\label{eq:BV_alpha_p_leibniz_C_b_ibp}
\int_{\R^n}
F\cdot \nabla^\alpha g\di x
=
-
\int_{\R^n}
g\di \divm^\alpha F.
\end{equation}
\end{theorem}

\begin{proof}
We mimic the proof of~\cite{Comi-Stefani22-L}*{Th.~3.1}. 
Since $g\in L^q(\R^n)\cap L^\infty(\R^n)$, we have $gF \in L^1(\R^n)\cap L^p(\R^n)$ by H\"older's inequality. 
In addition, \cite{Comi-Stefani22-L}*{Cor.~2.7} implies that $\div^{\alpha}_{\rm NL}(g, F) \in L^1(\R^n)$.
We now divide the proof in two steps.

\smallskip

\textit{Step~1: proof of~\eqref{eq:BV_alpha_p_leibniz_C_b}}.
Let $\xi \in\Lip_c(\R^n)$ be given.
By \cite{Comi-Stefani22-L}*{Lem.~3.2(i)}, we have
\begin{equation*}
\nabla^\alpha(g\xi)
=
g\,\nabla^\alpha \xi
+
\xi \nabla^\alpha g
+\nabla_{\rm NL}^\alpha(g,\xi)
\quad
\text{in}\
L^q(\R^n),
\end{equation*} 
so that
\begin{align*}
\int_{\R^n}g F \cdot \nabla^\alpha \xi \di x 
=
\int_{\R^n}F \cdot \nabla^\alpha(g\xi) \di x
-
\int_{\R^n}
\xi F \cdot\nabla^\alpha g
\di x
-
\int_{\R^n}
F \cdot \nabla^\alpha_{\rm NL}(g,\xi)\di x.
\end{align*}
By \cite{Comi-Stefani22-L}*{Lem.~2.9}, we have that
\begin{equation*}
\int_{\R^n}
F \cdot \nabla^\alpha_{\rm NL}(g,\xi)\,
dx
=
\int_{\R^n}
\xi \, \div^\alpha_{\rm NL}(g,F)\,
dx.
\end{equation*}
Now let $(F_\eps)_{\eps>0}\subset \DM^{\alpha,p}(\R^n)\cap C^\infty(\R^n; \R^n)$ be given by $F_\eps=\rho_\eps*F$ for all $\eps>0$. 
In particular, we have $F_{\eps} \in W^{1, p}(\R^n; \R^n)$ for each $\eps > 0$. Note that $W^{1,p}(\R^n; \R^n) \subset B^{\alpha}_{p,q}(\R^n; \R^n)$ for all $\alpha \in (0,1)$ and $p, q \in [1, +\infty]$, see~\cite{Leoni17}*{Th.~17.33}. 
As a consequence,  $F_\eps\in B^\alpha_{p,1}(\R^n; \R^n)$ for each $\eps>0$.
Since $g\xi \in B^\alpha_{q,1}(\R^n)$, by \cite{Comi-Stefani22-L}*{Lem.~2.6} we can write
\begin{equation*}
\int_{\R^n}F_\eps \cdot
\nabla^\alpha(g\xi)
\di x
=
-
\int_{\R^n}g\xi \,
\div^\alpha F_\eps\di x
\end{equation*}
for all $\eps>0$.
On the one side, we have
\begin{equation*}
\lim_{\eps\to0^+}
\int_{\R^n}F_\eps \cdot \nabla^\alpha(g\xi)
\di x
=
\int_{\R^n}
F \cdot \nabla^\alpha(g\xi)
\di x
\end{equation*}
by H\"older's inequality in the case $p<+\infty$ and by the Dominated Convergence Theorem in the case $p=+\infty$.
On the other side, since $g\xi \in C_c(\R^n)$, we also have
\begin{equation*}
\lim_{\eps\to0^+}
\int_{\R^n}g\xi \, \div^\alpha F_\eps\di x
=
\int_{\R^n}g\xi \di \divm^\alpha F,
\end{equation*}
thanks to \cref{res:approx_DM}.
We thus conclude that 
\begin{equation*}
\int_{\R^n}F \cdot \nabla^\alpha(g\xi)
\di x
=
-
\int_{\R^n}g\xi \di\divm^\alpha F,
\end{equation*}
so that, for all $\xi\in\Lip_c(\R^n)$,
\begin{equation*}
\int_{\R^n}g F \cdot \nabla^\alpha\xi\di x 
=
-
\int_{\R^n}g\xi \di\divm^\alpha F
-
\int_{\R^n}
\xi\, F \cdot\nabla^\alpha g
\di x
-
\int_{\R^n}
\xi \, \div^\alpha_{\rm NL}(g,F)\di x.
\end{equation*}
By a standard approximation argument for the test function, we get~\eqref{eq:BV_alpha_p_leibniz_C_b}. 

\smallskip

\textit{Step~2: proof of~\eqref{eq:BV_alpha_p_leibniz_C_b_D_alpha_zero} and~\eqref{eq:BV_alpha_p_leibniz_C_b_ibp}}.
Since $g F \in \DM^{\alpha,1}(\R^n)$ by Step~1, the first equation in~\eqref{eq:BV_alpha_p_leibniz_C_b_D_alpha_zero} readily follows from \cref{res:zero_total_mesure}.
Moreover, since obviously $\nabla^\alpha_{\rm NL}(g,v)=0$ for all $v\in\R$,
by~\cite{Comi-Stefani22-L}*{Lem.~2.9} we get
\begin{align*}
v
\int_{\R^n}\div^\alpha_{\rm NL}(g,F)\di x
=
\int_{\R^n}v \, \div^\alpha_{\rm NL}(g,F)\di x
=
\int_{\R^n}F \cdot \nabla^\alpha_{\rm NL}(g,v)\di x
=
0
\end{align*} 
for all $v\in\R$ and also the second equation in~\eqref{eq:BV_alpha_p_leibniz_C_b_D_alpha_zero} immediately follows.
By combining~\eqref{eq:BV_alpha_p_leibniz_C_b} with~\eqref{eq:BV_alpha_p_leibniz_C_b_D_alpha_zero}, we get~\eqref{eq:BV_alpha_p_leibniz_C_b_ibp} and the proof is complete.  
\end{proof}

We are now in position to prove our second main result \cref{res:leibniz_rules}.

\begin{proof}[Proof of \cref{res:leibniz_rules}]
The proofs of the cases $p\in\left[1,\frac n{n-\alpha}\right)$, $p\in\left[\frac n{1-\alpha},+\infty\right)$ and $p = + \infty$ are analogous to those of \cite{Comi-Stefani22-L}*{Cors.~3.3, 3.6 and~3.7}, respectively, and are hence omitted. 
We thus deal with the case $p\in\left[\frac n{n-\alpha},\frac n{1-\alpha}\right)$.
We start by noticing that 
$\gamma_{n,p,\alpha} \ge \alpha$ if and only if $p \ge \frac{n}{n-\alpha}$,
so that $B^{\gamma}_{q,1}(\R^n) \subset B^{\alpha}_{q,1}(\R^n)$, thanks to \cite{Leoni17}*{Th.~17.82}. Hence $g \in B^{\alpha}_{q, 1}(\R^n)$ and so $\nabla^\alpha g \in L^q(\R^n; \R^n)$ by \cite{Comi-Stefani22-L}*{Cor.~23 and Lem.~2.6}.
Let $(\rho_\eps)_{\eps > 0}$ be as in \cref{res:approx_DM} and set $g_\eps=\rho_\eps*g$ for all $\eps>0$. 
Arguing as in the proof of~\cite{Comi-Stefani22-L}*{Cor.~3.5}, we can exploit~\cite{Comi-et-al21}*{Sec.~5.1 and Th.~11} to conclude that 
\begin{equation}
\label{eq:g_eps_limit_quasievery_point}
\lim_{\eps\to0^+}
g_\eps(x)
=
g^\star (x)
\quad
\text{for all}\ x\in\R^n\setminus D_g,
\end{equation}
for some set $D_g\subset\R^n$ such that $\Haus{n-\gamma q+\delta}(D_g) = 0$ for any $\delta>0$ sufficiently small. 
Since
\begin{equation*}
n - \frac{nq}{n + (1-\alpha)q} > n - \gamma q \iff 
\gamma > \frac{n}{n + (1-\alpha)q},
\end{equation*}
we conclude that $|\divm^\alpha F|(D_g) = 0$, by \cref{res:abs_frac_div}. Since $g_\eps\in  C_b(\R^n)\cap B^\alpha_{q,1}(\R^n)$ for all $\eps>0$ thanks to \cite{Leoni17}*{Prop.~17.12}, by \cref{res:BV_alpha_p_leibniz_C_b} we get that $g_\eps F \in \DM^{\alpha,p}(\R^n)$, with
\begin{equation*}
\divm^{\alpha}(g_\eps F) 
= 
g_\eps \, \divm^{\alpha} F 
+ 
F\cdot
\nabla^{\alpha}g_\eps \,
\Leb{n} 
+ 
\div^{\alpha}_{\rm NL}(g_\eps, F)\,\Leb{n}
\quad 
\text{in}\ \M (\R^n).
\end{equation*}
Now
$\nabla^\alpha g_\eps
=
\rho_\eps*\nabla^\alpha g$ 
in
$L^q(\R^n;\R^n)$
(for example see~\cite{Comi-Stefani19}*{Lem.~3.5} and its proof), while \cite{Comi-Stefani22-L}*{Cor.~2.7} implies that 
\begin{align*}
\|\div^{\alpha}_{\rm NL}(g_\eps, F)
-
\div^{\alpha}_{\rm NL}(g, F)\|_{L^1(\R^n)}
&=
\|\div^{\alpha}_{\rm NL}(g_\eps-g, F)
\|_{L^1(\R^n)}
\\
&\le
2\mu_{n,\alpha}
\,
\|F\|_{L^p(\R^n; \R^n)}
\,
[g-g_\eps]_{B^\alpha_{q,1}(\R^n)}
\end{align*} 
for all $\eps>0$. 
Therefore, since $\rho_\eps*\nabla^\alpha g \to \nabla^{\alpha} g$ in $L^{q}(\R^n; \R^n)$ and, by~\cite{Leoni17}*{Prop.~17.12}, $[g-g_\eps]_{B^\alpha_{q,1}(\R^n)} \to 0$, the conclusion follows by exploiting~\eqref{eq:g_eps_limit_quasievery_point} and the Dominated Convergence Theorem with respect to the measure $|\divm^\alpha F|$. Finally, equations~\eqref{eq:ping} and~\eqref{eq:pong} can be proved as \eqref{eq:BV_alpha_p_leibniz_C_b_D_alpha_zero} and \eqref{eq:BV_alpha_p_leibniz_C_b_ibp} in \cref{res:BV_alpha_p_leibniz_C_b}.
\end{proof}

\section{Examples}
\label{sec:examples}

In this section, we illustrate some examples concerning \cref{res:abs_frac_div}.

\subsection{Example for \texorpdfstring{point~\eqref{item:abs_frac_div_subcritical} of \cref{res:abs_frac_div}}{point (i) in Theorem 1.6}}
\label{subsec:examples}

\cref{exa:F_frac} below  shows that, if $p \in \left [1, \frac{n}{n - \alpha} \right )$, the fractional divergence-measure of $\DM^{\alpha, p}$ vector fields is not absolutely continuous with respect to $\Haus{\eps}$ for any $\eps > 0$, in general, proving point~\eqref{item:abs_frac_div_subcritical} of \cref{res:abs_frac_div}.

\begin{example}
\label{exa:F_frac}
Let $\alpha \in (0, 1)$, $y, z \in \R^n$, and define 
\begin{equation}
\label{eq:esempio_frac}
F_{y, z, \alpha}(x) = \mu_{n, -\alpha} \left (\frac{(x - y)}{|x - y|^{n + 1 - \alpha}} -  \frac{(x - z)}{|x - z|^{n + 1 - \alpha}} \right ),
\quad
x\in\R^n\setminus\set*{y,z}.
\end{equation}
A plain computation yields $F_{y, z, \alpha} \in L^p(\R^n ; \R^n)$ for all $p \in \left [1, \frac{n}{n - \alpha} \right)$ (for example, see the proof of~\cite{Comi-Stefani19}*{Prop.~3.14}).
Moreover, by~\cite{Comi-Stefani19}*{Prop.~3.13}, we know that 
\begin{equation*}
\divm^\alpha F_{y, z, \alpha} = \delta_y - \delta_z.
\end{equation*}
Consequently, $F_{y, z, \alpha} \in \DM^{\alpha, p}(\R^n)$ for all $p \in \left [1, \frac{n}{n - \alpha} \right )$. 
\end{example}

Interestingly, the vector field~\eqref{eq:esempio_frac} of \cref{exa:F_frac} works also in the limit case $\alpha = 1$, proving point~\eqref{item:classical_subcritical} of \cref{res:abs_div_classical}, see~\cite{Silhavy05}*{Prop.~6.1}.

\begin{example}
\label{exa:F_int}
Let $y,z\in\R^n$ and define
\begin{equation*}
F_{y, z, 1}(x) = \mu_{n, -1} \left (\frac{(x - y)}{|x - y|^{n}} -  \frac{(x - z)}{|x - z|^{n}} \right ),
\quad
x\in\R^n\setminus\set*{y,z}.
\end{equation*}
Computations as in \cref{exa:F_frac} show that $F_{y, z, 1} \in L^p(\R^n ; \R^n)$ for all $p \in \left ( 1, \frac{n}{n - 1} \right )$, with 
\begin{equation*}
\divm F_{y, z, 1} = \delta_y - \delta_z.
\end{equation*}
Hence $F_{y, z, 1} \in \DM^{1, p}(\R^n)$ for all $p \in \left ( 1, \frac{n}{n - 1} \right )$. 
Actually, we have $F_{y, z, 1} \in \DM^{1, 1}_{\rm loc}(\R^n)$.
\end{example}

\subsection{Partial sharpness of \texorpdfstring{\cref{res:abs_frac_div}}{Theorem 1.6}}
\label{subsec:sharp}

Arguing as in~\cite{Silhavy05}*{Exam.~3.3 and Prop.~6.1}, we can exploit the properties of the vector field~\eqref{eq:esempio_frac} in \cref{exa:F_frac} to construct additional examples proving a partial sharpness of \cref{res:abs_frac_div}.

The following result is the analogue of~\cite{Comi-et-al21}*{Prop.~5}.

\begin{proposition}[The vector field $G_\alpha=F_\alpha*\nu$]
\label{res:F_alpha}
Let $\alpha\in(0,1)$ and $F_\alpha = F_{0, {\rm e}_1,\alpha}$ be as in \cref{exa:F_frac}, and let $\nu\in\M(\R^n)$.
Then we have 
\begin{equation*}
G_\alpha=F_\alpha*\nu \in \DM^{\alpha,p}(\R^n) 
\quad
\text{for all}\ p \in \left [1, \frac{n}{n - \alpha} \right ),	
\end{equation*}
with 
\begin{equation} \label{eq:fract_der_nu}
\divm^{\alpha}\, G_{\alpha} = \nu - (\tau_{{\rm e}_1})_{\#} \nu,
\end{equation}
where 
$\tau_{x}(y) = y + x$ for all $x,y\in\R^n$. In addition, if there exist $C, \eps > 0$ such that
\begin{equation}
\label{eq:mu_radius}
|\nu|(B_r(x))
\le
Cr^\eps
\quad\text{for all $x\in\R^n$ and $r>0$},
\end{equation}
then 
\begin{equation}
\label{eq:u_alpha_integrability}
G_\alpha \in \DM^{\alpha,p}(\R^n)
\quad
\text{for all}\
p\in
\begin{cases}
\left [1, \frac{n - \eps}{n - \alpha - \eps}\right )
& 
\text{if}\ \eps \in (0, n - \alpha),
\\[1.5ex]
[1,+\infty)
& 
\text{if}\ \eps = n - \alpha,
\\[1.5ex]
[1,+\infty]
& 
\text{if}\ \eps \in (n - \alpha,n].\\[2mm]
\end{cases}
\end{equation}
\end{proposition}

\begin{proof}
We divide the proof into two steps.

\smallskip

\textit{Step~1}.
Let $\nu\in\mathscr M(\R^n)$. 
We claim that $G_{\alpha}\in \DM^{\alpha, p}(\R^n)$ for all 
$p \in \left [1, \frac{n}{n - \alpha} \right )$ and that $G_{\alpha}$ satisfies \eqref{eq:fract_der_nu}.
Indeed, by Young's inequality (for Radon measures), we can estimate
\begin{equation*}
\|G_{\alpha}\|_{L^{1}(\R^n; \R^n)} 
\le 
\|F_{\alpha}\|_{L^{1}(\R^n; \R^n)} |\nu|(\R^n).
\end{equation*}
Moreover, thanks to the translation invariance of $\nabla^{\alpha}$ and exploiting the explicit expression of $F_\alpha$ given in \cref{exa:F_frac}, we can write 
\begin{align*}
\int_{\R^n} G_{\alpha}(x) \cdot \nabla^{\alpha} \xi(x) \di x 
& = 
\int_{\R^n} \int_{\R^n} F_{\alpha}(x - y) \cdot \nabla^{\alpha} \xi (x) \di  \nu(y) \di x
\\ 
& = 
\int_{\R^n} \int_{\R^n} F_{\alpha}(x - y) \cdot \nabla^{\alpha} \xi (x) \di x \di  \nu(y) 
\\
& = 
- \int_{\R^n} \int_{\R^n} \xi(x + y) \di  \left ( \delta_{0}(x) - \delta_{{\rm e}_1}(x) \right ) \di  \nu(y) \\
& = 
- \int_{\R^n} \left ( \xi(y) - \xi(y + {\rm e}_1) \right ) \di  \nu(y)
\end{align*}
for all $\xi \in C^\infty_c(\R^n)$.
This proves $G_{\alpha} \in \DM^{\alpha, 1}(\R^n)$ and~\eqref{eq:fract_der_nu}.
In addition, by Jensen's inequality and Tonelli's Theorem, we can estimate
\begin{align*}
\int_{\R^n} |G_{\alpha}(x)|^{p} \di x 
& \le \int_{\R^n}
|\nu|(\R^n)^{p - 1}
\int_{\R^n} |F_{\alpha}(x - y)|^{p} \di  |\nu|(y) \di x 
\\
&= |\nu|(\R^n)^{p} \,\|F_{\alpha}\|_{L^{p}(\R^n; \R^n)}^p 
< + \infty
\end{align*}
for all $p\in\left[1,\frac{n}{n-\alpha}\right)$, thanks to the integrability properties of $F_\alpha$ given in \cref{exa:F_frac}.

\smallskip

\textit{Step~2}.
We prove that~\eqref{eq:mu_radius} implies~\eqref{eq:u_alpha_integrability}.
To this aim, let $q=\frac p{p-1}$ and $0 < \delta \le q$.
Since 
$|F_{\alpha}| = |F_{\alpha}|^{\frac{\delta}{q}} |F_{\alpha}|^{1 - \frac{\delta}{q}}$,
by H\"older's inequality we get
\begin{align*}
|G_{\alpha}(x)|^{p} 
& \le 
\left ( \int_{\R^n} |F_{\alpha}(x - y)|^{\frac{\delta}{q}}\, |F_{\alpha}(x - y)|^{1 - \frac{\delta}{q}} \di  |\nu|(y) \right )^{p} 
\\
& \le 
\left ( \int_{\R^n} |F_{\alpha}(x - y)|^{\delta} \di  |\nu|(y) \right )^{\frac{p}{q}} 
\left ( \int_{\R^n} |F_{\alpha}(x - y)|^{p \left (1 - \frac{\delta}{q} \right )} \di  |\nu|(y) \right )
\end{align*}
for a.e.\ $x\in\R^n$.
We now recall the explicit expression of $F_\alpha$ in \cref{exa:F_frac} and write
\begin{equation}
\label{eq:u_alpha_splitting}
\begin{split}
\int_{\R^n} |F_{\alpha}(x - y)|^{\delta} \di  |\nu|(y) 
& = 
\int_{\R^n \setminus \left (B_{\frac{1}{2}}(x) \cup B_{\frac{1}{2}}(x - {\rm e}_1) \right )} |F_{\alpha}(x - y)|^{\delta} \di  |\nu|(y) 
\\
& \quad + \sum_{j = 1}^{\infty} \int_{C_{j}\left (x, \frac{1}{2} \right ) \cup C_{j}\left (x - {\rm e}_1, \frac{1}{2} \right )} |F_{\alpha}(x - y)|^{\delta} \di  |\nu|(y),
\end{split}
\end{equation}
where we have set
\begin{equation*}
C_{j}(x, r) = B_{r^j}(x) \setminus B_{r^{j+1}}(x) 
\end{equation*}
for all $x\in\R^n$, $r\in(0,1)$ and $j\in\N$, $j \ge 1$, for brevity.
Now, on the one hand, if 
$y \in \R^n \setminus \left (B_{\frac{1}{2}}(x) \cup B_{\frac{1}{2}}(x - {\rm e}_1) \right )$, 
then $x - y \in \R^n \setminus \left (B_{\frac{1}{2}} \cup B_{\frac{1}{2}}({\rm e}_1) \right )$, so that
\begin{equation*}
|F_{\alpha}(x - y)| \le \mu_{n, - \alpha} \left ( 2^{n - \alpha} + 2^{n - \alpha} \right ) = \mu_{n, - \alpha}\, 2^{n+1 - \alpha}
\end{equation*}
for all 
$y \in \R^n \setminus \left (B_{\frac{1}{2}}(x) \cup B_{\frac{1}{2}}(x - {\rm e}_1) \right )$.
Therefore, we can estimate
\begin{equation}
\label{eq:u_alpha_split_est1}
\int_{\R^n \setminus \left (B_{\frac{1}{2}}(x) \cup B_{\frac{1}{2}}(x - {\rm e}_1) \right )} |F_{\alpha}(x - y)|^{\delta} \di  |\nu|(y) 
\le 
\left (\mu_{n, - \alpha}\, 2^{n+1 - \alpha} \right )^{\delta} |\nu|(\R^n)
\end{equation}
for all $x\in\R^n$.
On the other hand, 
for all $x \in \R^n$ and $j \ge 1$, 
we have
\begin{align}
\int_{C_{j}\left (x, \frac{1}{2} \right)} |F_{\alpha}(x - y)|^{\delta} \di  |\nu|(y) 
& \le 
\mu_{n, - \alpha}^{\delta} \int_{C_{j}\left (x, \frac{1}{2} \right)} \left ( |x - y|^{\alpha - n} + |x - y - {\rm e}_1|^{\alpha - n} \right )^{\delta} \di  |\nu|(y) 
\nonumber\\
& \le \mu_{n, - \alpha}^{\delta} \left ( 2^{(j + 1)(n - \alpha)} + \left ( 1 - 2^{-j} \right )^{\alpha - n} \right )^{\delta} |\nu|( B_{2^{-j}}(x) )  
\nonumber\\
& \le \label{eq:u_alpha_split_est2} 
\mu_{n, - \alpha}^{\delta} \left ( 2^{(j + 1)(n - \alpha)} + 2^{n - \alpha} \right )^{\delta} C \, 2^{- j \eps}.
\end{align}
Reasoning analogously, we obtain
\begin{equation}
\label{eq:u_alpha_split_est3}
\int_{C_{j}\left (x- {\rm e}_1, \frac{1}{2} \right)} |F_{\alpha}(x - y)|^{\delta} \di  |\nu|(y) \le C \mu_{n, - \alpha}^{\delta} \left ( 2^{(j + 1)(n - \alpha)} + 2^{n - \alpha} \right )^{\delta} 2^{- j \eps}
\end{equation}
for all $x \in \R^n$ and $j \ge 1$.
Therefore, inserting~\eqref{eq:u_alpha_split_est1}, \eqref{eq:u_alpha_split_est2} and~\eqref{eq:u_alpha_split_est3} in~\eqref{eq:u_alpha_splitting}, we get that 
\begin{equation} 
\label{eq:f_alpha_delta_bound}
\int_{\R^n} |F_{\alpha}(x - y)|^{\delta} \di  |\nu|(y) 
\le 
C_{\alpha, \eps, \delta}
\end{equation}
for all $x\in\R^n$, where $C_{\alpha, \eps, \delta}>0$ is constant depending on~$\alpha$, $\eps$, and~$\delta$ which is finite provided that we choose $\delta <\frac\eps{n-\alpha}$, as we are assuming from now on.
We thus have
\begin{align*}
\int_{\R^n} |G_{\alpha}(x)|^{p} \di x 
& \le 
C_{\alpha, \eps, \delta}^{p - 1}\, 
\int_{\R^n} \int_{\R^n}
|F_{\alpha}(x - y)|^{p \left (1 - \frac{\delta}{q} \right )} \di  |\nu|(y) \di x 
\\
& = 
C_{\alpha, \eps, \delta}^{p - 1} 
\,
|\nu|(\R^n) 
\int_{\R^n} |F_{\alpha}(x)|^{p \left (1 - \frac{\delta}{q} \right )} \di x.
\end{align*}
Now, recalling \cref{exa:F_frac}, we immediately see that
\begin{equation*}
\int_{\R^n} |F_{\alpha}(x)|^{p \left (1 - \frac{\delta}{q} \right )} \di x < +\infty
\iff
p < \frac{n}{(n - \alpha) (1 - \delta)} - \frac{\delta}{1 - \delta} = \frac{n - \delta n + \alpha \delta}{(n - \alpha) (1 - \delta)}.
\end{equation*}
Hence, since the function
$
\delta\mapsto\frac{n - \delta n + \alpha \delta}{(n - \alpha) (1 - \delta)}
$
is monotone increasing, we easily see that
\begin{equation*}
\eps \in (0, n - \alpha)
\implies
\delta < \frac{\eps}{n - \alpha} < 1
\implies
p \in \left [1, \frac{n - \eps}{n - \alpha - \eps}\right)
\end{equation*}
and, similarly,
\begin{equation*}
\eps \in [n - \alpha, n]
\implies
\delta (n - \alpha) < \eps\
\text{for all}\ 
\delta\in(0,1)
\implies
p \in [1, +\infty).
\end{equation*}
Finally, in the case $\eps \in (n - \alpha, n]$, we exploit \eqref{eq:f_alpha_delta_bound} for $\delta = 1$ in order to conclude that
\begin{equation*}
|G_{\alpha}(x)| 
\le 
\int_{\R^n} |F_{\alpha}(x - y)| \di  |\nu|(y) 
= 
C_{\alpha, \eps} 
< 
+\infty
\end{equation*}
for all $x\in\R^n$, which implies that $G_{\alpha} \in L^{\infty}(\R^n)$.
The conclusion thus follows.
\end{proof}

Thanks to \cref{res:F_alpha}, we can now give the following example.

\begin{example}
\label{exa:falconer}
Let $\alpha\in(0,1)$ and let $\nu$ and $G_\alpha$ be as in \cref{res:F_alpha}.
By~\cite{Falconer14}*{Cor.~4.12}, there exists a compact set $K\subset\R$ such that $\nu = \Haus{\eps} \res K$, so that $|\divm^{\alpha} G_{\alpha}|\not\ll\Haus{s}$ for all $s > \eps$. 
Now we observe that, by~\eqref{eq:u_alpha_integrability}, we have the following situations:
\begin{itemize}
\item in order to have $G_\alpha \in \DM^{\alpha, p}(\R^n)$ for some $p \in \left [ \frac{n}{n-\alpha}, + \infty \right )$, we need $\eps > n - \alpha q$, since, if $\eps \in [n - \alpha, n]$, then $p \in [1, + \infty)$, while, for $\eps \in (0, n - \alpha)$, we have $p < \frac{n - \eps}{n - \alpha - \eps}$, which implies $\eps > n - \alpha q$;

\item in order to have $G_\alpha \in \DM^{\alpha, \infty}(\R^n)$, we need $\eps > n - \alpha$, since, if $\eps \in (n - \alpha, n]$, then $p \in [1, + \infty]$.
\end{itemize}
Therefore, these lower bounds on $\eps$ imply that, for $p \in \left [ \frac{n}{n-\alpha}, + \infty \right ]$, we have 
\begin{equation*}
|\divm^{\alpha} G_{\alpha}|\not\ll\Haus{s} \text{ for all } s > n - \alpha q. 
\end{equation*}
\end{example}

Notice that
\begin{equation*}
n - \alpha q 
\ge 
\max
\set*{
n -  \frac{nq}{nq + (1-\alpha)q} 
,\
\frac{n}{q} - \alpha
}
\end{equation*}
for all $q \in \left [1, \frac{n}{\alpha} \right ]$, which means $p \in \left [ \frac{n}{n-\alpha}, + \infty \right ]$, with equality only for $q = \frac{n}{\alpha}$ and $q = 1$.
Consequently, \cref{exa:falconer} shows that points~\eqref{item:abs_frac_div_intermediate} and~\eqref{item:abs_frac_div_supercritical} in \cref{res:abs_frac_div} cannot be improved beyond $|\divm^\alpha F| \ll \Haus{n - \alpha q}$, which is actually sharp for $p = + \infty$. 

\begin{remark}[Correction to~\cite{Comi-et-al21}*{Exam.~2}]
For $n=1$, \cref{exa:falconer} together with the above considerations corrects the conclusions of \cite{Comi-et-al21}*{Exam.~2}.
\end{remark}



\begin{bibdiv}
\begin{biblist}

\bib{Ambrosio-Crippa-Maniglia05}{article}{
      author={Ambrosio, Luigi},
      author={Crippa, Gianluca},
      author={Maniglia, Stefania},
       title={Traces and fine properties of a {$BD$} class of vector fields and
  applications},
        date={2005},
        ISSN={0240-2963},
     journal={Ann. Fac. Sci. Toulouse Math. (6)},
      volume={14},
      number={4},
       pages={527--561},
}


\bib{AFP00}{book}{
   author={Ambrosio, Luigi},
   author={Fusco, Nicola},
   author={Pallara, Diego},
   title={Functions of bounded variation and free discontinuity problems},
   series={Oxford Mathematical Monographs},
   publisher={The Clarendon Press, Oxford University Press, New York},
   date={2000},
}

\bib{Anzellotti84}{article}{
   author={Anzellotti, Gabriele},
   title={Pairings between measures and bounded functions and compensated compactness},
   journal={Ann. Mat. Pura Appl. (4)},
   volume={135},
   date={1983},
   pages={293--318 (1984)},
}

\bib{Brue-et-al20}{article}{
author={Bru{\`e}, Elia},
   author={Calzi, Mattia},
   author={Comi, Giovanni E.},
   author={Stefani, Giorgio},
   title={A distributional approach to fractional Sobolev spaces and
   fractional variation: asymptotics II},
   journal={C. R. Math. Acad. Sci. Paris},
   volume={360},
   date={2022},
   pages={589--626},
}

\bib{brue2022constancy}{article}{
  title={Constancy of the dimension in codimension one and locality of the unit normal on ${\rm RCD}(K, N)$ spaces},
  author={Bru{\'e}, Elia},
  author={Pasqualetto, Enrico}, 
  author={Semola, Daniele},
  journal={Ann. Sc. Norm. Super. Pisa Cl. Sci.},
  year={2022}
}

\bib{BuffaComiMira}{article}{
   author={Buffa, Vito},
   author={Comi, Giovanni E.},
   author={Miranda, Michele Jr.},
   title={On BV functions and essentially bounded divergence-measure fields
   in metric spaces},
   journal={Rev. Mat. Iberoam.},
   volume={38},
   date={2022},
   number={3},
   pages={883--946},
}

\bib{Chen-et-al19}{article}{
   author={Chen, Gui-Qiang G.},
   author={Comi, Giovanni E.},
   author={Torres, Monica},
   title={Cauchy fluxes and Gauss-Green formulas for divergence-measure fields over general open sets},
   journal={Arch. Ration. Mech. Anal.},
   volume={233},
   date={2019},
   number={1},
   pages={87--166},
}

\bib{Chen-Frid99}{article}{
   author={Chen, Gui-Qiang},
   author={Frid, Hermano},
   title={Divergence-measure fields and hyperbolic conservation laws},
   journal={Arch. Ration. Mech. Anal.},
   volume={147},
   date={1999},
   number={2},
   pages={89--118},
}

\bib{Chen-Frid01}{article}{
      author={Chen, Gui-Qiang},
      author={Frid, Hermano},
       title={On the theory of divergence-measure fields and its applications},
        date={2001},
     journal={Bol. Soc. Bras. Mat.},
      volume={32},
      number={3},
       pages={401--433},
}

\bib{Chen-Frid03}{article}{
      author={Chen, Gui-Qiang},
      author={Frid, Hermano},
       title={Extended divergence-measure fields and the {E}uler equations for
  gas dynamics},
        date={2003},
        ISSN={0010-3616},
     journal={Comm. Math. Phys.},
      volume={236},
      number={2},
       pages={251--280},
}

\bib{ChTo}{article}{
      author={Chen, Gui-Qiang},
      author={Torres, Monica},
       title={On the structure of solutions of nonlinear hyperbolic systems of
  conservation laws},
        date={2011},
     journal={Commun. Pure Appl. Anal.},
      volume={10},
      number={4},
       pages={1011\ndash 1036},
}


\bib{Cicalese-Trombetti03}{article}{
      author={Cicalese, Marco},
      author={Trombetti, Cristina},
       title={Asymptotic behaviour of solutions to {$p$}-{L}aplacian equation},
        date={2003},
        ISSN={0921-7134},
     journal={Asymptot. Anal.},
      volume={35},
      number={1},
       pages={27--40},
}

\bib{ComiPhD}{article}{
   author={Comi, Giovanni E.},
   title={Refined Gauss--Green formulas and evolution problems for Radon measures},
   journal={Scuola Normale Superiore, Pisa},
   date={2020},
 note={Ph.D.\ Thesis. Available at \href{https://cvgmt.sns.it/paper/4579/}{cvgmt.sns.it/paper/4579/}}
}

\bib{comi2022representation}{article}{
  author={Comi, Giovanni E.} 
  author={Crasta, Graziano},
  author={De Cicco, Virginia},
  author={Malusa, Annalisa},
  title={Representation formulas for pairings between divergence-measure fields and $BV$ functions},
   note={Preprint, available at \href{https://arxiv.org/abs/2208.10812}{arXiv:2208.10812}},
  year={2022},
}

\bib{Comi-Payne20}{article}{
   author={Comi, Giovanni E.},
   author={Payne, Kevin R.},
   title={On locally essentially bounded divergence measure fields and sets
   of locally finite perimeter},
   journal={Adv. Calc. Var.},
   volume={13},
   date={2020},
   number={2},
   pages={179--217},
}

\bib{ComiMagna}{article}{
   author={Comi, Giovanni E.},
   author={Magnani, Valentino},
   title={The Gauss-Green theorem in stratified groups},
   journal={Adv. Math.},
   volume={360},
   date={2020},
   pages={106916, 85},
}

\bib{Comi-et-al21}{article}{
  author={Comi, Giovanni E.},
   author={Spector, Daniel},
   author={Stefani, Giorgio},
   title={The fractional variation and the precise representative of
   $BV^{\alpha,p}$ functions},
   journal={Fract. Calc. Appl. Anal.},
   volume={25},
   date={2022},
   number={2},
   pages={520--558},
}

\bib{Comi-Stefani19}{article}{
   author={Comi, Giovanni E.},
   author={Stefani, Giorgio},
   title={A distributional approach to fractional Sobolev spaces and
   fractional variation: Existence of blow-up},
   journal={J. Funct. Anal.},
   volume={277},
   date={2019},
   number={10},
   pages={3373--3435},
}

\bib{Comi-Stefani22-A}{article}{
   author={Comi, Giovanni E.},
   author={Stefani, Giorgio},
   title={A distributional approach to fractional Sobolev spaces and fractional variation: asymptotics I},
 journal={Rev. Mat. Complut.},
   date={2022},
}

\bib{Comi-Stefani22-L}{article}{
  author={Comi, Giovanni E.},
  author={Stefani, Giorgio},
   title={Leibniz rules and Gauss-Green formulas in distributional
   fractional spaces},
   journal={J. Math. Anal. Appl.},
   volume={514},
   date={2022},
   number={2},
   pages={Paper No. 126312, 41},
}

\bib{Comi-Stefani23}{article}{
   author={Comi, Giovanni E.},
   author={Stefani, Giorgio},
    title={Failure of the local chain rule for the fractional variation},
 journal={Port. Math.},
  volume={},
    date={2023},
   number={},
   pages={}, 
}

\bib{Comi-Stefani23-On}{article}{
   author={Comi, Giovanni E.},
   author={Stefani, Giorgio},
    title={On sets with finite distributional fractional perimeter},
    date={2023},
note={Preprint, available at \href{https://arxiv.org/abs/2303.10989}{arXiv:2303.10989}},
}

\bib{Crasta-DeCicco19-Anzellotti}{article}{
   author={Crasta, Graziano},
   author={De Cicco, Virginia},
   title={Anzellotti's pairing theory and the Gauss-Green theorem},
   journal={Adv. Math.},
   volume={343},
   date={2019},
   pages={935--970},
}

\bib{Crasta-DeCicco19-An}{article}{
      author={Crasta, Graziano},
      author={De Cicco, Virginia},
       title={An extension of the pairing theory between divergence-measure
  fields and {BV} functions},
        date={2019},
        ISSN={0022-1236},
     journal={J. Funct. Anal.},
      volume={276},
      number={8},
       pages={2605--2635},
         url={https://doi.org/10.1016/j.jfa.2018.06.007},
      review={\MR{3926127}},
}

\bib{CD5}{article}{
author={Crasta, Graziano},
      author={De Cicco, Virginia},
       title={On the variational nature of the Anzellotti pairing},
       note={Preprint, available at \href{https://arxiv.org/abs/2207.06469}{arXiv:2207.06469}},
  year={2022},
}

\bib{Crasta-et-al22}{article}{
   author={Crasta, Graziano},
   author={De Cicco, Virginia},
   author={Malusa, Annalisa},
   title={Pairings between bounded divergence-measure vector fields and BV
   functions},
   journal={Adv. Calc. Var.},
   volume={15},
   date={2022},
   number={4},
   pages={787--810},
}

\bib{MR3978950}{article}{
   author={De Cicco, Virginia},
   author={Giachetti, Daniela},
   author={Oliva, Francescantonio},
   author={Petitta, Francesco},
   title={The Dirichlet problem for singular elliptic equations with general
   nonlinearities},
   journal={Calc. Var. Partial Differential Equations},
   volume={58},
   date={2019},
   number={4},
   pages={Paper No. 129, 40},
}


\bib{MR3939259}{article}{
   author={De Cicco, Virginia},
   author={Giachetti, Daniela},
   author={Segura de Le\'{o}n, Sergio},
   title={Elliptic problems involving the 1-Laplacian and a singular lower
   order term},
   journal={J. Lond. Math. Soc. (2)},
   volume={99},
   date={2019},
   number={2},
   pages={349--376},
}
		

\bib{DeLellis-Otto-Westdickenberg03}{article}{
      author={De~Lellis, Camillo},
      author={Otto, Felix},
      author={Westdickenberg, Michael},
       title={Structure of entropy solutions for multi-dimensional scalar
  conservation laws},
        date={2003},
        ISSN={0003-9527},
     journal={Arch. Ration. Mech. Anal.},
      volume={170},
      number={2},
       pages={137--184},
}

\bib{Degiovanni-et-al99}{article}{
      author={Degiovanni, Marco},
      author={Marzocchi, Alfredo},
      author={Musesti, Alessandro},
       title={Cauchy fluxes associated with tensor fields having divergence
  measure},
        date={1999},
        ISSN={0003-9527},
     journal={Arch. Ration. Mech. Anal.},
      volume={147},
      number={3},
       pages={197--223},
}

\bib{DiNezza-et-al12}{article}{
   author={Di Nezza, Eleonora},
   author={Palatucci, Giampiero},
   author={Valdinoci, Enrico},
   title={Hitchhiker's guide to the fractional Sobolev spaces},
   journal={Bull. Sci. Math.},
   volume={136},
   date={2012},
   number={5},
   pages={521--573},
}

\bib{Falconer14}{book}{
   author={Falconer, Kenneth},
   title={Fractal geometry},
   edition={3},
   note={Mathematical foundations and applications},
   publisher={John Wiley \& Sons, Ltd., Chichester},
   date={2014},
}

\bib{Frid12}{article}{
      author={Frid, Hermano},
       title={Remarks on the theory of the divergence-measure fields},
        date={2012},
        ISSN={0033-569X},
     journal={Quart. Appl. Math.},
      volume={70},
      number={3},
       pages={579--596},
}

\bib{Frid14}{article}{
   author={Frid, Hermano},
   title={Divergence-measure fields on domains with Lipschitz boundary},
   conference={
      title={Hyperbolic conservation laws and related analysis with
      applications},
   },
   book={
      series={Springer Proc. Math. Stat.},
      volume={49},
      publisher={Springer, Heidelberg},
   },
   date={2014},
   pages={207--225},
}

\bib{Grafakos14-C}{book}{
   author={Grafakos, Loukas},
   title={Classical Fourier analysis},
   series={Graduate Texts in Mathematics},
   volume={249},
   edition={3},
   publisher={Springer, New York},
   date={2014},
}

\bib{Grafakos14-M}{book}{
   author={Grafakos, Loukas},
   title={Modern Fourier analysis},
   series={Graduate Texts in Mathematics},
   volume={250},
   edition={3},
   publisher={Springer, New York},
   date={2014},
}

\bib{Kawhol-Schuricht07}{article}{
      author={Kawohl, Bernd},
      author={Schuricht, Friedemann},
       title={Dirichlet problems for the 1-{L}aplace operator, including the
  eigenvalue problem},
        date={2007},
        ISSN={0219-1997},
     journal={Commun. Contemp. Math.},
      volume={9},
      number={4},
       pages={515--543},
}

\bib{MR4241342}{article}{
   author={Latorre, Marta},
   author={Oliva, Francescantonio},
   author={Petitta, Francesco},
   author={Segura de Le\'{o}n, Sergio},
   title={The Dirichlet problem for the 1-Laplacian with a general singular
   term and $L^1$-data},
   journal={Nonlinearity},
   volume={34},
   date={2021},
   number={3},
   pages={1791--1816},
}

\bib{leonardi2023prescribed}{article}{
  title={The prescribed mean curvature measure equation in non-parametric form},
  author={Leonardi, Gian Paolo}
  author={Comi, Giovanni E.},
   note={Preprint, available at \href{https://arxiv.org/abs/2302.10592}{arXiv:2302.10592}},
  year={2023},
}

\bib{Leonardi-Saracco22}{article}{
   author={Leonardi, Gian Paolo},
   author={Saracco, Giorgio},
   title={Rigidity and trace properties of divergence-measure vector fields},
   journal={Adv. Calc. Var.},
   volume={15},
   date={2022},
   number={1},
   pages={133--149},
}

\bib{Leoni17}{book}{
   author={Leoni, Giovanni},
   title={A first course in Sobolev spaces},
   series={Graduate Studies in Mathematics},
   volume={181},
   edition={2},
   publisher={American Mathematical Society, Providence, RI},
   date={2017},
}

\bib{Liu-Xiao22}{article}{
      author={Liu, Liguang},
      author={Xiao, Jie},
       title={Divergence \& curl with fractional order},
        date={2022},
        ISSN={0021-7824},
     journal={J. Math. Pures Appl. (9)},
      volume={165},
       pages={190--231},
}

\bib{MR2502520}{article}{
   author={Mercaldo, Anna},
   author={Segura de Le\'{o}n, Sergio},
   author={Trombetti, Cristina},
   title={On the solutions to 1-Laplacian equation with $L^1$ data},
   journal={J. Funct. Anal.},
   volume={256},
   date={2009},
   number={8},
   pages={2387--2416},
}

\bib{Phuc-Torres08}{article}{
   author={Phuc, Nguyen Cong},
   author={Torres, Monica},
   title={Characterizations of the existence and removable singularities of
   divergence-measure vector fields},
   journal={Indiana Univ. Math. J.},
   volume={57},
   date={2008},
   number={4},
   pages={1573--1597},
}

\bib{Ponce-Spector20}{article}{
   author={Ponce, Augusto C.},
   author={Spector, Daniel},
   title={A boxing inequality for the fractional perimeter},
   journal={Ann. Sc. Norm. Super. Pisa Cl. Sci. (5)},
   volume={20},
   date={2020},
   number={1},
   pages={107--141},
}


\bib{Scheven-Schmidt16}{article}{
   author={Scheven, Christoph},
   author={Schmidt, Thomas},
   title={BV supersolutions to equations of 1-Laplace and minimal surface
   type},
   journal={J. Differential Equations},
   volume={261},
   date={2016},
   number={3},
   pages={1904--1932},
}

\bib{MR3813962}{article}{
   author={Scheven, Christoph},
   author={Schmidt, Thomas},
   title={On the dual formulation of obstacle problems for the total
   variation and the area functional},
   journal={Ann. Inst. H. Poincar\'{e} C Anal. Non Lin\'{e}aire},
   volume={35},
   date={2018},
   number={5},
   pages={1175--1207},
}

\bib{Schu}{article}{
      author={Schuricht, Friedemann},
       title={A new mathematical foundation for contact interactions in
  continuum physics},
        date={2007},
     journal={Arch. Ration. Mech. Anal.},
      volume={184},
      number={3},
       pages={495\ndash 551},
}



\bib{Silhavy05}{article}{
   author={{\v{S}}ilhav{\'{y}}, Miroslav},
   title={Divergence measure fields and Cauchy's stress theorem},
   journal={Rend. Sem. Mat. Univ. Padova},
   volume={113},
   date={2005},
   pages={15--45},
}

\bib{Silhavy07}{article}{
      author={{\v{S}}ilhav{\'{y}}, Miroslav},
       title={Divergence measure vectorfields: their structure and the
  divergence theorem},
        date={2007},
   conference={
title={Mathematical modelling of bodies with complicated bulk and
  boundary behavior},
},
book={
      series={Quad. Mat.},
      volume={20},
   publisher={Dept. Math., Seconda Univ. Napoli, Caserta},
},
       pages={217--237},
}

\bib{Silhavy09}{article}{
      author={{\v{S}}ilhav{\'{y}}, Miroslav},
       title={The divergence theorem for divergence measure vectorfields on
  sets with fractal boundaries},
        date={2009},
     journal={Math. Mech. Solids},
      volume={14},
      number={5},
       pages={445--455},
}

\bib{Silhavy19}{article}{
  author={{\v{S}}ilhav{\'{y}}, Miroslav},
 journal={	
Indiana Univ. Math. J.},
date={2019},
note={To appear.},
title={The Gauss-Green theorem for bounded vectorfields with divergence measure on sets of finite perimeter},
}

\bib{Silhavy20}{article}{
      author={{\v{S}}ilhav{\'{y}}, Miroslav},
       title={Fractional vector analysis based on invariance requirements
  (critique of coordinate approaches)},
        date={2020},
        ISSN={0935-1175},
     journal={Contin. Mech. Thermodyn.},
      volume={32},
      number={1},
       pages={207--228},
}

\bib{Silhavy22}{article}{
      author={{\v{S}}ilhav{\'{y}}, Miroslav},
       title={Fractional strain tensor and fractional elasticity},
        date={2022},
     journal={J. Elast.},
}

\bib{Stein70}{book}{
   author={Stein, Elias M.},
   title={Singular integrals and differentiability properties of functions},
   series={Princeton Mathematical Series, No. 30},
   publisher={Princeton University Press, Princeton, N.J.},
   date={1970},
}

\bib{Stein93}{book}{
   author={Stein, Elias M.},
   title={Harmonic analysis: real-variable methods, orthogonality, and oscillatory integrals},
   series={Princeton Mathematical Series},
   volume={43},
   publisher={Princeton University Press, Princeton, NJ},
   date={1993},
}

\end{biblist}
\end{bibdiv}
\end{document}